\documentclass[11pt]{extarticle}
\usepackage{amsfonts, amsmath, amssymb, amsthm,float} 

\usepackage{charter} % Font
\usepackage[square,numbers]{natbib} % Bibliography

\usepackage{graphicx,geometry,titlesec,authblk}
\usepackage{xcolor} % Colour names [dvipsnames]

\usepackage{tikz} % Maths diagrams
\usetikzlibrary{shapes.geometric}

\usepackage{hyperref,cleveref} 
\definecolor{theme}{RGB}{170, 0, 0}

\titleformat{\section}  % Which section command to format
  {\centering\fontsize{12}{16}\bfseries} % Format for whole line
  {\thesection} % How to show number
  {1em} % Space between number and text
  {} % Formatting for just the text
  [] % Formatting for after the text

\titleformat{\subsection}  % Which section command to format
  {\fontsize{12}{16}\bfseries} % Fformat for whole line
  {\thesubsection} % How to show number
  {1em} % Space between number and text
  {} % Formatting for just the text
  [] % Formatting for after the text

\hypersetup{
    colorlinks=true,
    linkcolor=black, % Colour of internal links
    filecolor=black,
    citecolor=black, % Colour of internal citations
    urlcolor=black,
    pdftitle={On Universal Graphs for Trees and Treewidth $k$ Graphs}, % Title in pdf view
    pdfpagemode=FullScreen,
    }

\DeclareMathOperator{\tw}{tw}

\DeclareMathOperator{\C}{\mathcal{C}}
\DeclareMathOperator{\bigo}{\mathcal{O}}
\renewcommand{\leq}{\leqslant}
\renewcommand{\geq}{\geqslant}
\usepackage{footmisc} % Foot note
\newtheorem{theorem}{Theorem}

\newtheorem{lemma}[theorem]{Lemma}
\newtheorem*{claim}{Claim}
\newtheorem{corollary}[theorem]{Corollary}
\newtheorem{fact}[theorem]{Fact}
\newcommand{\define}[1]{{\color{theme}\textit{#1}}} % New maths definition
\title{\bf\boldmath On Universal Graphs for Trees  and Treewidth $k$ Graphs}
\author{Neel Kaul and David R. Wood\footnote{Research of Wood is supported by the Australian Research Council and by NSERC.}}
\affil{School of Mathematics\\Monash University\\Melbourne, Australia\\\medskip
\texttt{nkau0020@student.monash.edu, david.wood@monash.edu}}

\begin{document}

\maketitle

\begin{abstract}
    Let $s(n)$ be the minimum number of edges in a graph that contains every $n$-vertex tree as a subgraph. Chung and Graham [\textit{J.~London Math.\ Soc.\ }1983] claim to prove that $s(n)\leq\bigo(n\log n)$. We point out a mistake in their proof. The previously best known upper bound is $s(n)\leq\bigo(n(\log n)(\log\log n)^{2})$ by Chung, Graham and Pippenger [\textit{Proc.\ Hungarian Coll.\ on Combinatorics} 1976], the proof of which is missing many crucial details. We give a fully self-contained proof of the new and improved upper bound $s(n)\leq\bigo(n(\log n)(\log\log n))$. The best known lower bound is $s(n)\geq\Omega(n\log n)$. 
    
    We generalise these results for graphs of treewidth $k$. For an integer $k\geq 1$, let $s_k(n)$ be the minimum number of edges in a graph that contains every $n$-vertex graph with treewidth $k$ as a subgraph. So $s(n)=s_1(n)$. We show 
    that $\Omega(k n\log n) \leq s_k(n) \leq \bigo(kn(\log n)(\log\log n))$.
\end{abstract}

%\tableofcontents

\section{Introduction}\label{sec: intro}

Consider the following classical `universal graph' problem: given a class of graphs $\mathcal{H}$ and an integer $n$, what is the minimum number of edges in a graph that contains\footnote{A graph $U$ \define{contains} a graph $G$ if a subgraph of $U$ is isomorphic to $G$.} every $n$-vertex graph in $\mathcal{H}$? A particular case of interest is when $\mathcal{H}$ is the class of trees.

Let $s(n)$ be the minimum number of edges in a graph that contains every $n$-vertex tree. The asymptotic behaviour of $s(n)$ is (up until now) largely considered to be well-understood. An $\Omega(n\log n)$ lower bound can by seen by observing that any graph $U$ that contains every $n$-vertex tree must contain, for each $i\in\{1,\ldots,n\}$, the forest with $i$ components, each of which is a star on $\lfloor\frac{n}{i}\rfloor$ vertices. Thus, the degree sequence of $U$ lexicographically dominates $(\lfloor\frac{n}{1}\rfloor-1,\lfloor\frac{n}{2}\rfloor-1,\lfloor\frac{n}{3}\rfloor-1,\ldots)$, implying $|E(U)|\geq\Omega(n \log n)$. 
\citet{GLST23} improved this lower bound by a multiple of 2.

In a series of heavily cited papers by \citet*{chung1978trees1}, \citet*{chung1978trees2}, and \citet*{chung1983universal}, a sequence of upper bounds on $s(n)$ were proved, culminating in a $\bigo(n\log n)$ bound by \citet*{chung1983universal}, asymptotically matching the $\Omega(n \log n)$ lower bound. The first contribution of this paper is to point out a fundamental error in the proof of this upper bound in \cite{chung1983universal}. In \Cref{sec: admissible graphs}, we introduce the necessary definitions from the original paper and explain the mistake in detail. 

The best known upper bound on $s(n)$ prior to the claimed $\bigo(n\log n)$ bound is $s(n)\leq\bigo(n(\log n)(\log\log n)^{2})$, due to \citet*{chung1978trees2}. The proof of this result has three major steps. The first involves showing the existence of a certain type of separator in a tree. The second involves using these separators to recursively construct a graph $G$ that contains every tree of a certain order. The third step involves estimating the number of edges in $G$. The proof given in \cite{chung1978trees2} is missing many crucial details, in particular, details of the first step are left to the reader, and the argument for the second step is incomplete. In 
\Cref{TreeUpperBound}, we build on their methods and present a fully self-contained proof of the new and improved upper bound $s(n)\leq\bigo(n(\log n)(\log\log n))$. This is now the best known upper bound on $s(n)$.

It is an open problem to asymptotically determine $s(n)$. We know
\begin{align*}
    \Omega(n\log n)\leq s(n)\leq\bigo(n(\log n)(\log\log n))\;.
\end{align*}

We generalise these results through the notion of treewidth, which is the standard measure of how similar a graph is to a tree, and is an important parameter in structural and algorithmic graph theory; see \citep{Reed97,Bodlaender98,HW17} for surveys.

A \define{tree-decomposition} of a graph $G$ is a collection of subsets $(B_{x}:x\in V(T))$ of $V(G)$ (called \define{bags}) indexed by the vertices of a tree $T$, such that:
\begin{enumerate}
    \item[(i)]\textit{(vertex-property)} for all $v\in V(G)$, $T[\{x\in V(T):v\in B_{x}\}]$ is a nonempty subtree of $T$, and
    \item[(ii)]\textit{(edge-property)} for all $uv\in E(G)$, there exists $x\in V(T)$ such that $\{u,v\}\subseteq B_{x}$.
\end{enumerate}
The \define{width} of a tree-decomposition $(B_{x}:x\in V(T))$ is $\max_{x\in V(T)}|B_{x}|-1$. The \define{treewidth} of $G$, denoted by $\tw(G)$, is the minimum width of a tree-decomposition of $G$.

For an integer $k\geq 1$, let $s_k(n)$ be the minimum number of edges in a graph that contains every graph with $n$ vertices and treewidth at most $k$. A graph has treewidth 1 if and only if it is a forest, so $s_1(n)=s(n)$. The best previously known upper bound on $s_k(n)$ is $\bigo(k n\log^{2}n)$, as explained by \citet{Joret}. We show the following bounds (where asymptotic notation means that $n\gg k$):
\begin{align*}
\Omega(k n\log n) \leq s_k(n) \leq \bigo(kn(\log n)(\log\log n))\;.
\end{align*}
Both the lower and upper bounds here are the best known (as far as we are aware). These results are proved in \Cref{Treewidthk}. It is an open problem to asymptotically determine $s_k(n)$. 

The proofs of our bounds on $s(n)$ and $s_{k}(n)$ give explicit non-optimised constants.

\section{Admissible graphs}\label{sec: admissible graphs}

The exact assertion that \citet{chung1983universal} claim to show is the inequality:
\begin{align*}
    s(n)\leq\frac{5}{\log 4}n\log n+\bigo(n)\;.
\end{align*}
where $\log$ means $\log_{e}$. In order to explain the mistake, we first present the necessary definitions from their proof, reusing their notation and terminology word-for-word as much as possible. \Cref{fig: binary tree 4 levels,fig: binary tree recurse,fig: admissible tree,fig: subcase b 1,fig: defining G2} are copied directly from \cite{chung1983universal}. Denote by $B(k)$ the \define{binary tree with $k$ levels} (see \Cref{fig: binary tree 4 levels}).

\begin{figure}[!h]
    \centering
    \includegraphics[width=0.7\linewidth]{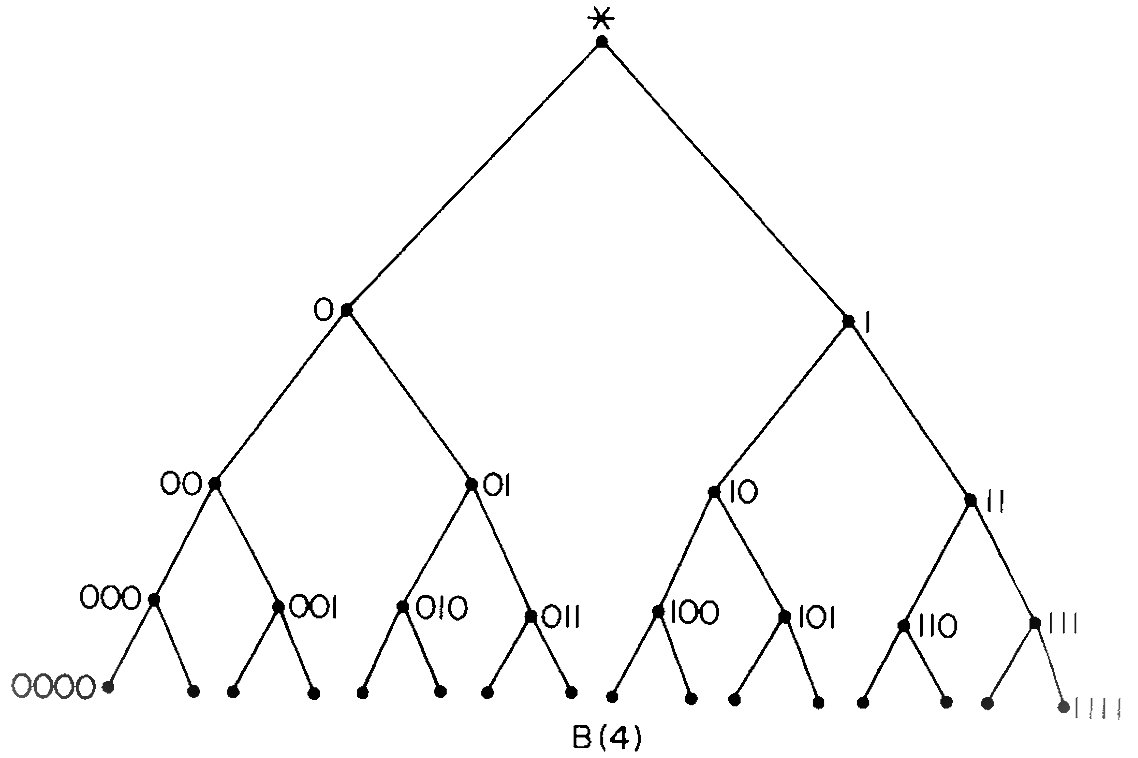}
    \caption{Binary tree with $k$ levels ($k=4$)}
    \label{fig: binary tree 4 levels}
\end{figure}

The vertices of $B(k)$ are the $\{0,1\}$-strings of length at most $k$ (and $*$ is the empty string). Denote the vertex set of $B(k)$ by $V(k)$. For $\alpha\in\{0,1\}$, denote by $B_{\alpha}(k-1)$ the copy of $B(k-1)$ rooted at vertex $\alpha$ (see \Cref{fig: binary tree recurse}).

\begin{figure}[H]
    \centering
    \includegraphics[width=0.5\linewidth]{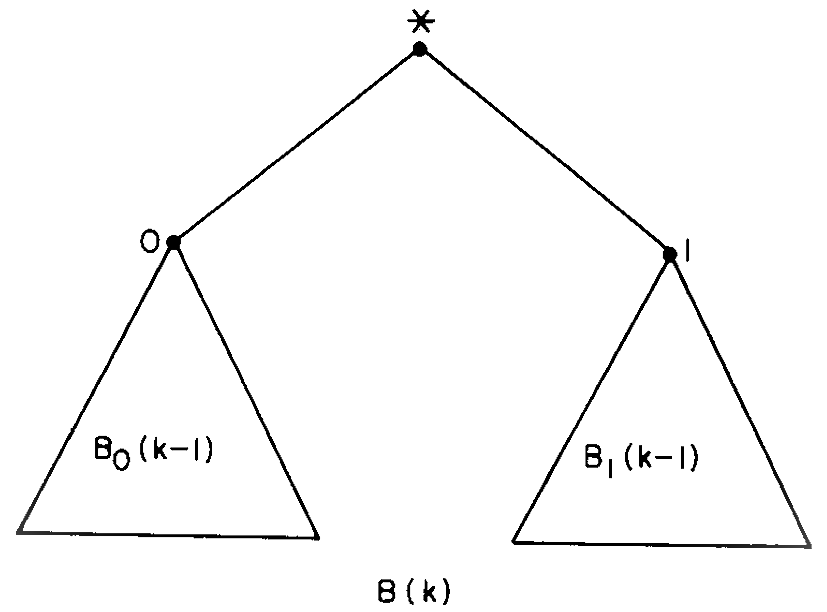}
    \caption{$B_{0}(k-1)$ and $B_{1}(k-1)$}
    \label{fig: binary tree recurse}
\end{figure}

If $\alpha=x_{1}\dots x_{j}\in V(k)$ with $j<k$, then $\alpha0$ and $\alpha1$ are the \define{left} and \define{right sons} of $\alpha$ respectively. Also, $\alpha$ is the \define{father} of $\alpha0$ and $\alpha1$, and $\alpha0$ is called the \define{left-hand brother} of $\alpha1$. A \define{descendant} of $\alpha$ is a vertex $\beta=y_{1}\dots y_{\ell}\in V(k)$ such that $j<\ell\leq k$ and for all $i\in\{1,\ldots,j\}$, $x_{i}=y_{i}$.

A subtree $A$ of $B(k)$ is said to be \define{admissible} if one of the following holds:
\begin{enumerate}
    \item[(i)]
    $V(A)=\varnothing$;
    
    \item[(ii)]
    $V(A)=\{*\}$;
    
    \item[(iii)] 
    $A$ is of the form (a) or (b) shown in \Cref{fig: admissible tree}, where $A'$ is an admissible subtree of $B(k-1)$ identified at its root to the indicated vertex.
\end{enumerate}

\begin{figure}[!t]
    \centering
    \includegraphics[width=0.75\linewidth]{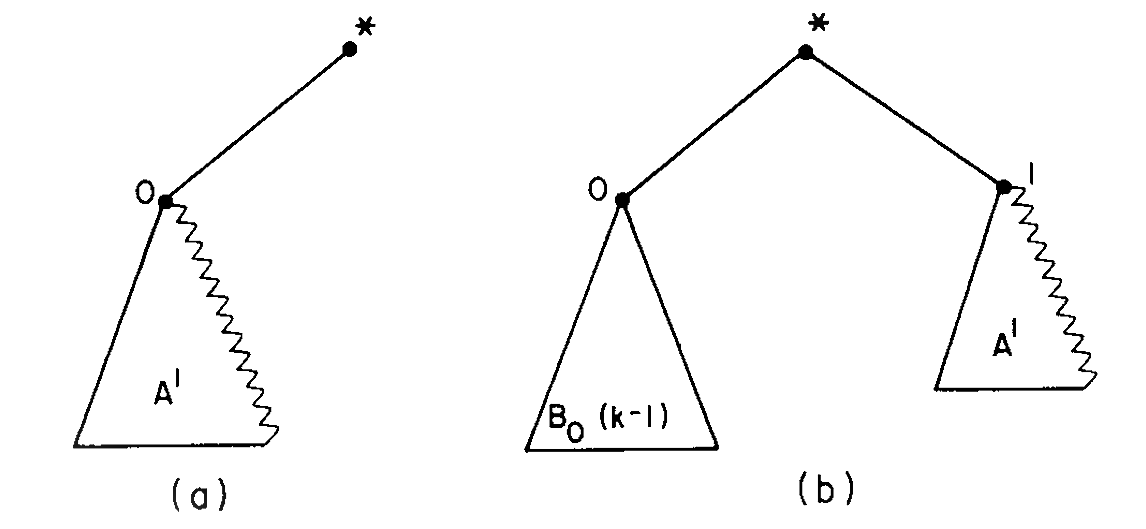}
    \caption{Structure of non-trivial admissible subtrees}
    \label{fig: admissible tree}
\end{figure}

For an admissible subtree $A$ of $B(k)$ and a vertex $v\in V(A)$, let \define{$\nu_{A}(v)$} denote the total number of descendants of $v$ in $V(A)$. Define a new graph $G(k)$ on the vertex set $V(k)$ with the following edges (see \Cref{fig: admissible graph}). For every $\alpha\in V(k)$, add an edge to $G(k)$ between $\alpha$ and:
\begin{enumerate}
    \item[(i)]
    each descendant of $\alpha$;
    
    \item[(ii)]
    the left-hand brother of $\alpha$ (if it exists) and each of its descendants;
    
    \item[(iii)] 
   the left-hand brother of the father of $\alpha$ (if it exists) and each of its descendants.
\end{enumerate}

\begin{figure}[!h]
    \centering
    \begin{tikzpicture}[/tikz/xscale=1, /tikz/yscale=1]
        % (i)
        \begin{scope}
            \draw[very thick] (-1,-0.5) -- (1,-0.5) -- (0,2) -- (-1,-0.5);
            \node (a) at (0,1.2) {};
            \node (b) at (0,0.9) {};
            \node (c) at (0,0.6) {};
            \node (d) at (0,0.3) {};
            \node (e) at (0,0) {};
        
            \node(index) at (0,-1.2) {(i)};
        \end{scope}
        
        \begin{scope}
            \node[draw, shape = circle, fill = black, minimum size = 0.1cm, inner sep=2pt] (alpha) at (0,2) {};
            \node (text) at (0.5,2) {$\alpha$};
        \end{scope}
        
        \begin{scope}[every edge/.style={draw=theme,very thick}]
            \path [-] (alpha) edge[bend right=50] (a);
            \path [-] (alpha) edge[bend right=60] (b);
            \path [-] (alpha) edge[bend right=70] (c);
            \path [-] (alpha) edge[bend right=80] (d);
            \path [-] (alpha) edge[bend right=90] (e);
        \end{scope}
    
        % (ii)
        \begin{scope}[shift={(4,0)}]
            \draw[very thick] (-1,-0.5) -- (1,-0.5) -- (0,2) -- (-1,-0.5);
            \node (a) at (0,1.2) {};
            \node (b) at (0,0.9) {};
            \node (c) at (0,0.6) {};
            \node (d) at (0,0.3) {};
            \node (e) at (0,0) {};
        
            \node[draw, shape = circle, fill = black, minimum size = 0.1cm, inner sep=2pt] (beta) at (0,2) {};
        
            \node[draw, shape = circle, fill = black, minimum size = 0.1cm, inner sep=2pt] (star) at (0.5,3) {};
        
            \node(index) at (0,-1.2) {(ii)};
        \end{scope}
        
        \begin{scope}[shift={(4,0)}]
            \node[draw, shape = circle, fill = black, minimum size = 0.1cm, inner sep=2pt] (alpha) at (1,2) {};
            \node (text) at (1.5,2) {$\alpha$};
        \end{scope}
        
        \begin{scope}[every edge/.style={draw,very thick}]
            \path [-] (star) edge (alpha);
            \path [-] (star) edge (beta);
        \end{scope}
        
        \begin{scope}[every edge/.style={draw=theme,very thick}]
            \path [-] (alpha) edge (beta);
            \path [-] (alpha) edge[bend left=20] (a);
            \path [-] (alpha) edge[bend left=30] (b);
            \path [-] (alpha) edge[bend left=40] (c);
            \path [-] (alpha) edge[bend left=50] (d);
            \path [-] (alpha) edge[bend left=60] (e);
        \end{scope}
    
        % (iii)
        \begin{scope}[shift={(8,0)}]
            \draw[very thick] (-1,-0.5) -- (1,-0.5) -- (0,2) -- (-1,-0.5);
            \node (a) at (0,1.2) {};
            \node (b) at (0,0.9) {};
            \node (c) at (0,0.6) {};
            \node (d) at (0,0.3) {};
            \node (e) at (0,0) {};
        
            \node[draw, shape = circle, fill = black, minimum size = 0.1cm, inner sep=2pt] (beta) at (0,2) {};
        
            \node[draw, shape = circle, fill = black, minimum size = 0.1cm, inner sep=2pt] (gamma) at (1,2) {};
            
            \node[draw, shape = circle, fill = black, minimum size = 0.1cm, inner sep=2pt] (star) at (0.5,3) {};
        
            \node(index) at (0,-1.2) {(iii)};
        \end{scope}
        
        \begin{scope}[shift={(8,0)}]
            \node[draw, shape = circle, fill = black, minimum size = 0.1cm, inner sep=2pt] (alpha) at (1.5,1) {};
            \node (text) at (2,1) {$\alpha$};
        \end{scope}
        
        \begin{scope}[every edge/.style={draw,very thick}]
            \path [-] (star) edge (alpha);
            \path [-] (star) edge (beta);
        \end{scope}
        
        \begin{scope}[every edge/.style={draw=theme,very thick}]
            \path [-] (alpha) edge (beta);
            \path [-] (alpha) edge[bend right=5] (a);
            \path [-] (alpha) edge[bend left=10] (b);
            \path [-] (alpha) edge[bend left=20] (c);
            \path [-] (alpha) edge[bend left=30] (d);
            \path [-] (alpha) edge[bend left=40] (e);
        \end{scope}
    \end{tikzpicture}
    
    \caption{The three types of edges in $G(k)$}
    \label{fig: admissible graph}
\end{figure}
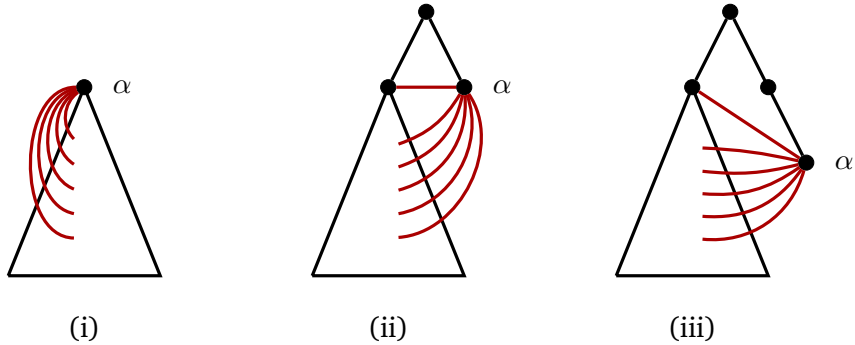

A subgraph $G$ of $G(k)$ is called \define{admissible} if it is an induced subgraph of  $G(k)$, and $V(G)$ is the vertex set of an admissible subtree of $B(k)$. For an admissible subgraph $G$ of $G(k)$ and a vertex $v\in V(G)$, let \define{$\nu_{G}(v)$} denote the total number of descendants of $v$ in $V(G)$. The main goal in \cite{chung1983universal} is to show that for all $n\geq1$, every $n$-vertex admissible graph $G$ contains every $n$-vertex tree. The following fundamental claim in the paper intends to achieve this goal. The proof of this claim is where the mistake lies.

\begin{claim}[{\cite[page 205]{chung1983universal}}]
    For every admissible graph $G$, for every tree $T$ with $|V(T)|\leq|V(G)|$ and for every $v\in V(T)$ there is an embedding\footnote{That is, an injective homomorphism.} $\lambda:V(T)\to V(G)$ such that $G-\lambda(V(T))$ is admissible and
    \begin{align*}
        \nu_{G}(\lambda(v))<|V(T)|\leq\nu_{G}(\lambda(v)^{*})\;,
    \end{align*}
    where $\lambda(v)^{*}$ denotes the father of $\lambda(v)$ in $G$ (if it exists).
\end{claim}

As presented in \cite{chung1983universal}, the proof of this claim proceeds by induction on the number of vertices in $G$ and $T$ (with priority given to $|V(G)|$). In the base case, $V(G)=\{*\}$. In the induction step, we are given an admissible graph $G$, a tree $T$ such that $|V(T)|\leq|V(G)|$, and a vertex $v\in V(T)$. Furthermore, $V(G)$ has the form (a) or (b) as shown in \Cref{fig: admissible tree}. Consider Subcase (b) 1. in the middle of page 207. Here, $G$ has the form shown in \Cref{fig: subcase b 1} (where $G'$ is an admissible subgraph of $G(k-2)$ identified at its root to vertex $10$) and $|V(T)|<|V(G)|$. 

\begin{figure}[h!]
    \centering
    \includegraphics[width=0.4\linewidth]{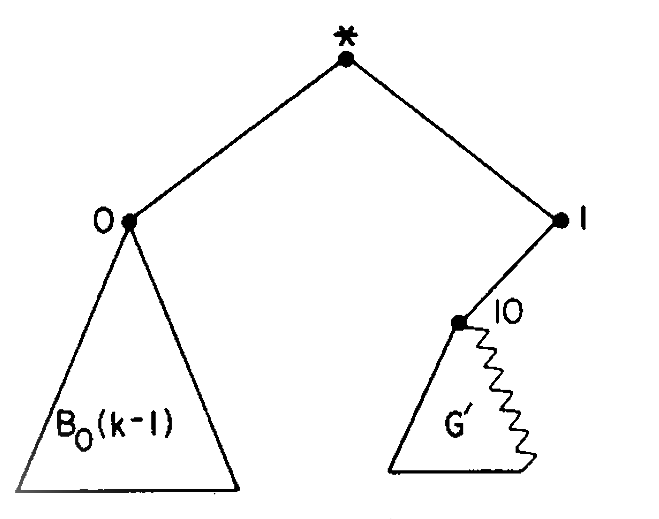}
    \caption{Structure of $G$ in Subcase (b) 1.}
    \label{fig: subcase b 1}
\end{figure}

Moreover, by the way the edges in $G(k)$ are defined, we can consider $G$ to have the form shown in \Cref{fig: defining G2}; let $G_{2}$ to be the subgraph of $G$ induced by the vertices in the circled region.

\begin{figure}[h!]
    \centering
    \includegraphics[width=0.6\linewidth]{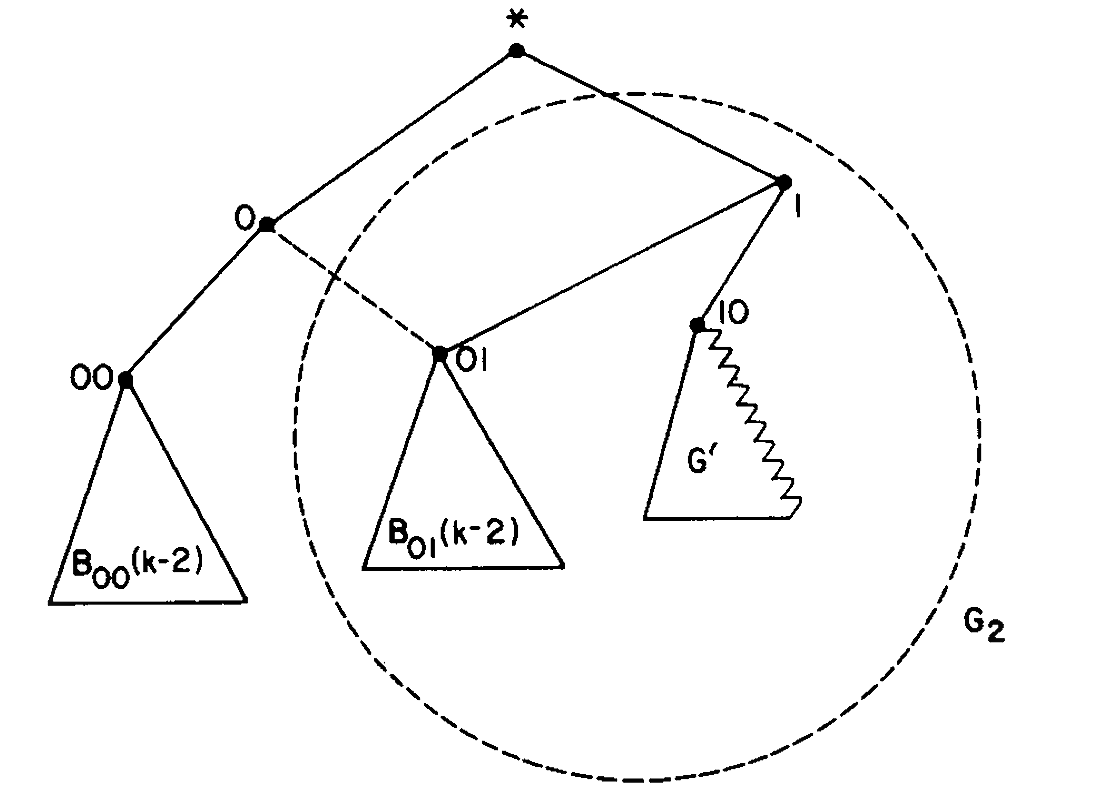}
    \caption{Defining $G_{2}$}
    \label{fig: defining G2}
\end{figure}

\citet[bottom of page 207]{chung1983universal}  claim that $G_{2}$ ``is (isomorphic to) an admissible graph". For the arguments that follow, it is vital that in the interpretation of the phrase ``isomorphic to" one assumes the isomorphism maps vertex $1$ of $G_{2}$ to the root $*$ of some admissible tree $A$. Consider the case when $2\leq|V(G')|<|V(B(k-2))|$. Then the first inequality implies that vertex $10$ has a son $\alpha$ in $G'$. By the structure of non-trivial admissible subtrees, the second inequality implies that $B_{01}(k-2)$ and $G'$ are mapped respectively to the left and right subtrees of the root of $A$. For $G_{2}$ to be admissible, by property (iii) of admissible graphs, $\alpha$ must be adjacent to the left-hand brother of the father of $\alpha$ and all of its descendants; that is, $\alpha$ is adjacent to vertex $01$ and each of its descendants. These edges are not present in $G$, and are therefore not present in $G_{2}$. Thus, we cannot conclude that $G_{2}$ is ``isomorphic to" an admissible graph, and so the proof of the claim does not follow through. This completes the explanation of the mistake.

\section{Upper Bound for Trees}
\label{TreeUpperBound}

This section presents our proof of the bound $s(n)\leq\bigo(n(\log n)(\log\log n))$. 

\subsection{$W$-subtrees of forests}\label{sec: w subtrees of forests}

For a forest $F$ and a set $W\subseteq V(F)$, a \define{$W$-subtree} of $F$ is a maximal subtree of $F$ in which the vertices in $W$ only appear as leaves (see \Cref{fig: w subtrees example}). Here, a \define{leaf} in a forest is a vertex of degree $1$. Note that every edge in $F$ lies in a unique $W$-subtree of $F$. 

\begin{figure}[!h]
    \centering
    \begin{tikzpicture}[/tikz/xscale=0.8, /tikz/yscale=0.8]
    \begin{scope}[every node/.style={circle,very thick,draw}]
        \node (a) at (0,0) {$a$};
        \node[draw=theme] (b) at (1.5,1.5) {$b$};
        \node[draw=theme] (c) at (2.5,-0.5) {$c$};
        \node (d) at (3.7,2) {$d$};
        \node[draw=theme] (e) at (5.7,1.5) {$e$};
        \node[draw=theme] (f) at (4.5,3.4) {$f$};
        \node (g) at (2,3) {$g$};
        \node (h) at (0,4) {$h$};
    \end{scope}
    
    \begin{scope}[every edge/.style={draw=black,very thick}]
        \path [-] (a) edge (b);
        \path [-] (b) edge (c);
        \path [-] (b) edge (d);
        \path [-] (d) edge (e);
        \path [-] (g) edge (h);
    \end{scope}
    
    \begin{scope}
        \node (W) at (2.25,-2.5) {$W=\{b,c,e,f\}$};
        \node (F) at (2.25,-1.7) {$F$};
    \end{scope}
    
    \begin{scope}[shift={(10,0)}, every node/.style={circle,very thick,draw}]
        \node (a) at (0-1,0) {$a$};
        \node[draw=theme] (b1) at (1.5-1,1.5) {$b$};
        \node[draw=theme] (b2) at (1.5,1.5-0.6) {$b$};
        \node[draw=theme] (b3) at (1.5,1.5+0.5) {$b$};
        \node[draw=theme] (c) at (2.5,-0.5-0.5) {$c$};
        \node (d) at (3.7,2+0.5) {$d$};
        \node[draw=theme] (e) at (5.7,1.5+0.5) {$e$};
        \node (g) at (2,3) {$g$};
        \node (h) at (0,4) {$h$};
    \end{scope}
    
    \begin{scope}[every edge/.style={draw=black,very thick}]
        \path [-] (a) edge (b1);
        \path [-] (b2) edge (c);
        \path [-] (b3) edge (d);
        \path [-] (d) edge (e);
        \path [-] (g) edge (h);
    \end{scope}
    
    \begin{scope}[shift={(10,0)}]
        \node (text) at (2.25,-2.2) {$W$-subtrees of $F$};
    \end{scope}
    \end{tikzpicture}

    \caption{A forest $F$ (left) and every $W$-subtree of $F$ (right)}
    \label{fig: w subtrees example}
\end{figure}
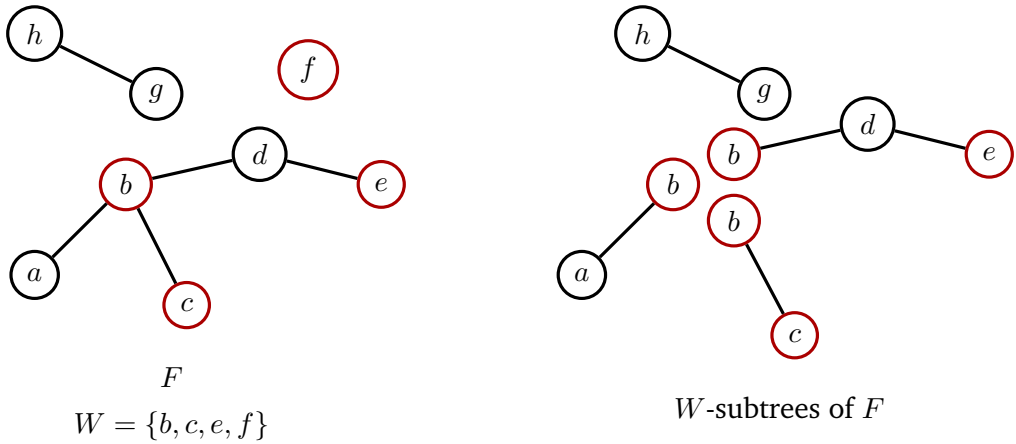

\noindent\citet{chung1978trees2} proved a slight variation of the following result about the existence of $W$-subtrees in large trees; their proof is missing details, so we include a proof for completeness.

\begin{lemma}\label{lem: w subtrees induction}
    For every integer $w\geq0$, for every real number $\alpha\geq3(\frac{3}{2})^{w}$, for every tree $T$ with $|E(T)|>\alpha$, there exists $W\subseteq V(T)$ with $|W|\leq w+1$ such that for some set $\C$ of $W$-subtrees of $T$,
    \begin{align*}
        \alpha<\sum_{T'\in\C}|E(T')|\leq\Big(1+\Big(\frac{2}{3}\Big)^{w}\Big)\alpha\;.
    \end{align*}
\end{lemma}

The proof this result requires the following lemma originally proved by \citet{chung1978trees1}; we include a proof for completeness. A \define{$v$-subtree} is a $\{v\}$-subtree.

\begin{lemma}\label{lem: w subtrees base case}
    For every real number $\alpha\geq1$, for every forest $F$ with $|E(F)|>\alpha$, there exists a vertex $v\in V(F)$ such that for some set $\C$ of $v$-subtrees of $F$,
    \begin{align*}
        \alpha<\sum_{T'\in\C}|E(T')|\leq2\alpha\;.
    \end{align*}
\end{lemma}
\begin{proof}
    If every component of $F$ has at most $\alpha$ edges, then pick any vertex $v\in V(F)$, and let $\{T_{1},\ldots,T_{j}\}$ be a minimal set of $v$-subtrees of $F$ such that $|E(T_{1})|+\cdots+|E(T_{j})|>\alpha$. If $|E(T_{1})|+\cdots+|E(T_{j})|>2\alpha$, then the minimality of $\{T_{1},\ldots,T_{j}\}$ implies
    \begin{align*}
        |E(T_{j})|=|E(T_{1})|+\cdots+|E(T_{j})|-(|E(T_{1})|+\cdots+|E(T_{j-1})|)>2\alpha-\alpha=\alpha\;,
    \end{align*}
    which contradicts that $|E(T_{j})|\leq\alpha$. Thus, $\{T_{1},\ldots,T_{j}\}$ is a desired set of $v$-subtrees of $F$.
    
    Now assume some component of $F$ has more than $\alpha$ edges, call it $T$. Let $x$ be a leaf in $T$. We may assume $|E(T)|>2\alpha$, otherwise $\{T\}$ is a desired set of $x$-subtrees of $F$. Then $|E(T-x)|>2\alpha-1\geq\alpha$. Let $y$ be the neighbour of $x$ in $T$. If every $y$-subtree of $T-x$ has at most $\alpha$ edges, then by the same argument given above, a minimal set of $y$-subtrees of $T-x$ whose union has more than $\alpha$ edges will suffice. Now assume some $y$-subtree of $T-x$ has more than $\alpha$ edges. Let $P=(v_{0},v_{1},\ldots,v_{t})$ be a maximal path in $T$ such that $v_{0}=x$, $v_{1}=y$ and for each $i\geq1$ there exists a $v_{i}$-subtree of $T-v_{i-1}$ that contains $v_{i}$ and has more than $\alpha$ edges. Let $T'$ be a $v_{t}$-subtree of $T-v_{t-1}$ that contains $v_{t}$ with $|E(T')|>\alpha$. We may assume $|E(T')|>2\alpha$, otherwise $\{T'\}$ is a desired set of $v_{t}$-subtrees of $F$. Then $|E(T'-v_{t})|>2\alpha-1\geq\alpha$. Let $v_{t+1}$ be the neighbour of $v_{t}$ in $T'$. By the maximality of $P$, every $v_{t+1}$-subtree in $T'-v_{t}$ has at most $\alpha$ edges. By the same argument given above, a minimal set of $v_{t+1}$-subtrees of $T'-v_{t}$ whose union has more than $\alpha$ edges will suffice. This proves the lemma.
\end{proof}

\begin{proof}[Proof of \Cref{lem: w subtrees induction}]
    Proceed by induction on $w$. The $w=0$ case is a direct application of \Cref{lem: w subtrees base case}. Assume the result holds for some $w\geq0$. Let $\alpha\geq3(\frac{3}{2})^{w+1}$ and let $T$ be a tree with $|E(T)|>\alpha$. Since $\alpha\geq3(\frac{3}{2})^{w}$, by the induction hypothesis there exists $W\subseteq V(T)$ with $|W|\leq w+1$ such that for some set $\C(W)$ of $W$-subtrees of $T$, $\alpha<\sum_{T'\in\C(W)}|E(T')|\leq(1+(\frac{2}{3})^{w})\alpha$. We are done if $\sum_{T'\in\C(W)}{|E(T')|}\leq(1+(\frac{2}{3})^{w+1})\alpha$, so assume the contrary. Then
    \begin{align}\label{ineq: 1 lem: w subtrees induction}
        \Big(1+\Big(\frac{2}{3}\Big)^{w+1}\Big)\alpha<\sum_{T'\in\C(W)}|E(T')|\leq\Big(1+\Big(\frac{2}{3}\Big)^{w}\Big)\alpha\;.
    \end{align}
    Let $F:=\bigcup_{T'\in\C(W)}T'$ and $\beta:=\frac{1}{3}(\frac{2}{3})^{w}\alpha\geq1$. Observe that $F$ is a forest. Since the subtrees in $\C(W)$ are edge-disjoint,
    \begin{align*}
        |E(F)|=\sum_{T'\in\C(W)}|E(T')|>\Big(1+\Big(\frac{2}{3}\Big)^{w+1}\Big)\alpha=\alpha+2\beta\geq\beta\;.
    \end{align*}
    Since $\beta\geq1$, \Cref{lem: w subtrees base case} implies that there is a vertex $v\in V(F)$ such that for some set $\C(v)$ of $v$-subtrees of $F$, 
    \begin{align}\label{ineq: 2 lem: w subtrees induction}
        \frac{1}{3}\Big(\frac{2}{3}\Big)^{w}\alpha=\beta<\sum_{T'\in\C(v)}|E(T')|\leq2\beta=\frac{2}{3}\Big(\frac{2}{3}\Big)^{w}\alpha\;.
    \end{align}
    Then \labelcref{ineq: 1 lem: w subtrees induction,ineq: 2 lem: w subtrees induction} together imply
    \begin{align}
        \alpha=\Big(1+\Big(\frac{2}{3}\Big)^{w+1}\Big)\alpha-\frac{2}{3}\Big(\frac{2}{3}\Big)^{w}\alpha&<\sum_{T'\in\C(W)}|E(T')|-\sum_{T'\in\C(v)}|E(T')|\label{ineq: 3 lem: w subtrees induction}\\
        &\leq\Big(1+\Big(\frac{2}{3}\Big)^{w}\Big)\alpha-\frac{1}{3}\Big(\frac{2}{3}\Big)^{w}\alpha\nonumber\\
        &=\Big(1+\Big(\frac{2}{3}\Big)^{w+1}\Big)\alpha\;.\nonumber
    \end{align}
    Let $\overline{\C(v)}$ be the set of all $v$-subtrees of $F$ not in $\C(v)$. Since $\{E(T'):T'\in\C(W)\}$ and $\{E(T'):T'\in\C(v)\}\cup\{E(T'):T'\in\overline{\C(v)}\}$ are both partitions of $E(F)$, 
    \begin{align*}
        |E(F)|=\sum_{T'\in\C(W)}|E(T')|=\sum_{T'\in\C(v)}|E(T')|+\sum_{T'\in\overline{\C(v)}}|E(T')|\;.
    \end{align*}
    Thus, \labelcref{ineq: 3 lem: w subtrees induction} implies
    \begin{align*}
        \alpha<\sum_{T'\in\overline{\C(v)}}|E(T')|\leq\Big(1+\Big(\frac{2}{3}\Big)^{w+1}\Big)\alpha\;.
    \end{align*}
    Hence, $\overline{\C(v)}$ is a desired set of $(W\cup\{v\})$-subtrees of $T$. The result follows by induction.
\end{proof}

\subsection{Useful functions}\label{sec: useful functions}

This section proves a number of elementary facts, which are used extensively in \Cref{sec: construction of Gwp,sec: estimating the edges in Gwp}.
%to show that a certain graph $G_{w,p}$ with few edges (depending on the parameters $w$ and $p$) contains every tree with at most some number of edges (depending on $w$ and $p$). 
For integers $w,p\geq0$, define
\begin{align*}
    &\alpha(w,p):=\frac{1}{1+(\frac{2}{3})^{w}}\Big(\frac{2+(\frac{2}{3})^{w}}{1+(\frac{2}{3})^{w}}\Big)^{p-1}\;,& &\beta(w):=3\Big(\frac{3}{2}\Big)^{w}\;,\\
    &\gamma(w,p):=\Big(\frac{2+(\frac{2}{3})^{w}}{1+(\frac{2}{3})^{w}}\Big)^{p}\;,\text{ and}& &\delta(w):=\Big(2+\Big(\frac{2}{3}\Big)^{w}\Big)\beta(w)\;.
\end{align*}
It follows immediately from the definition of $\alpha$ and $\gamma$ that
\begin{align}\label{eq: gamma and alpha relations}
    \gamma(w,p)=\Big(2+\Big(\frac{2}{3}\Big)^{w}\Big)\alpha(w,p)\qquad\text{and}\qquad\gamma(w,p-1)=\Big(1+\Big(\frac{2}{3}\Big)^{w}\Big)\alpha(w,p)\;.
\end{align}
For integers $w\geq0$, observe that for all sufficiently large integers $p$, $\alpha(w,p)\geq\frac{1}{2}(\frac{3}{2})^{p-1}\geq\beta(w)$. Define $p_{w}$ to be the smallest integer such that $\alpha(w,p_{w}+1)\geq\beta(w)$. Since $\alpha(w,p)$ is an increasing function of $p$, if $p\geq p_{w}+1$ then $\alpha(w,p)\geq\beta(w)$, and if $p\leq p_{w}$ then $\alpha(w,p)<\beta(w)$.

\begin{fact}\label{fact: pw is linear in w}
     For every integer $w\geq0$, $\log_{2}(\frac{3}{2})\cdot w+\log_{2}(3)\leq p_{w}\leq w+\log_{3/2}(9)$.
\end{fact}
\begin{proof}
    Observe that $p_{w}=\lceil x\rceil$ where $\alpha(w,x+1)=\beta(w)$. Rearranging $\alpha(w,x+1)=\beta(w)$ for $x$ gives 
    \begin{align*}
        x=\frac{\log\Big((1+(\frac{2}{3})^{w})\cdot3(\frac{3}{2})^{w}\Big)}{\log\Big(\frac{2+(\frac{2}{3})^{w}}{1+(\frac{2}{3})^{w}}\Big)}\;.
    \end{align*}
    Let $a=\log((1+(\frac{2}{3})^{w})\cdot3(\frac{3}{2})^{w})$ and $b=\log\Big(\frac{2+(\frac{2}{3})^{w}}{1+(\frac{2}{3})^{w}}\Big)$. Then, $a>\log(3(\frac{3}{2})^{w})$ and $b\leq\log(2)$, implying $\log_{2}(\frac{3}{2})\cdot w+\log_{2}(3)\leq \frac{a}{b}\leq p_{w}$. Furthermore,
    \begin{align*}
        a\leq\log\Big(\Big(1+\Big(\frac{2}{3}\Big)^{0}\Big)\cdot3\Big(\frac{3}{2}\Big)^{w}\Big)=\log\Big(6\Big(\frac{3}{2}\Big)^{w}\Big)\quad\text{and}\quad b\geq\log\Big(\frac{2+(\frac{2}{3})^{0}}{1+(\frac{2}{3})^{0}}\Big)=\log\Big(\frac{3}{2}\Big)\;,
    \end{align*}
    implying $\frac{a}{b}\leq w+\log_{3/2}(6)$. Therefore, $p_{w}\leq w+\log_{3/2}(6)+1=w+\log_{3/2}(9)$.
\end{proof}

\begin{fact}\label{fact: gamma >= alpha + 3}
    For all integers $w\geq0$ and $p\geq p_{w}+1$, $\gamma(w,p)\geq\alpha(w,p)+3$.
\end{fact}
\begin{proof}
    By the definition of $p_{w}$, $\alpha(w,p)\geq\beta(w)\geq3$ for all $w\geq0$ and $p\geq p_{w}+1$. By \labelcref{eq: gamma and alpha relations},
    \begin{align*}
        \gamma(w,p)-\alpha(w,p)=\Big(2+\Big(\frac{2}{3}\Big)^{w}\Big)\alpha(w,p)-\alpha(w,p)=\Big(1+\Big(\frac{2}{3}\Big)^{w}\Big)\alpha(w,p)\geq3\;.
    \end{align*}
    Therefore, $\gamma(w,p)\geq\alpha(w,p)+3$.
\end{proof}

\begin{fact}\label{fact: gamma < delta}
     For every integer $w\geq0$, $\gamma(w,p_{w})<\delta(w)$.
\end{fact}
\begin{proof}
    By the definition of $p_{w}$, $\alpha(w,p_{w})<\beta(w)$. By \labelcref{eq: gamma and alpha relations},
    \[\gamma(w,p_{w})=\Big(2+\Big(\frac{2}{3}\Big)^{w}\Big)\alpha(w,p_{w})<\Big(2+\Big(\frac{2}{3}\Big)^{w}\Big)\beta(w)=\delta(w)\;.\qedhere\]
\end{proof}

\begin{fact}\label{fact: delta upper bound}
     For every integer $w\geq0$, $\lceil \delta(w)\rceil\leq12(\frac{3}{2})^{w}$.
\end{fact}
\begin{proof}
    By the definition of $\delta(w)$ and $\beta(w)$,
    \[\lceil \delta(w)\rceil=\Big\lceil\Big(2+\Big(\frac{2}{3}\Big)^{w}\Big)\beta(w)\Big\rceil\leq\lceil3\beta(w)\rceil\leq4\beta(w)=12\Big(\frac{3}{2}\Big)^{w}\;.\qedhere\]
\end{proof}

\begin{fact}\label{fact: exponential of pw}
     For every integer $w\geq0$, $2^{-p_{w}}\leq\frac{1}{3}(\frac{2}{3})^{w}$.
\end{fact}
\begin{proof}
    By \Cref{fact: pw is linear in w}, $\log_{2}(\frac{3}{2})\cdot w+\log_{2}(3)\leq p_{w}$. Therefore,
    \[2^{-p_{w}}\leq2^{-\log_{2}(\frac{3}{2})\cdot w-\log_{2}(3)}=\frac{1}{3}\Big(\frac{2}{3}\Big)^{w}\;.\qedhere\]
\end{proof}

\subsection{The universal graph $G_{w,p}$}\label{sec: construction of Gwp}

For integers $w\geq0$ and $p\geq p_{w}$, define the graph $G_{w,p}$ recursively. Let $G_{w,p_{w}}:=K_{\lceil \delta(w)\rceil}$ and for $p\geq p_{w}+1$, let $G_{w,p}$ be the graph formed by adding all the possible edges between a complete graph $K_{w+1}$ and two disjoint copies of $G_{w,p-1}$, as shown in \Cref{fig: constructing Gwp}.

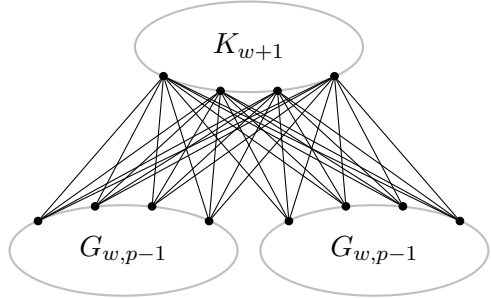
\begin{figure}[!h]
    \centering
    \begin{tikzpicture}[/tikz/xscale=1.5, /tikz/yscale=1.5]

        % Blobs
        \begin{scope}[every node/.style={ellipse, minimum width=3cm, minimum height=1.2cm, thick,draw, lightgray}]
            \node (complete) at (0,1.5) {\textcolor{black}{$K_{w+1}$}};
            \node (G1) at (-1.1,-0.3) {\textcolor{black}{$G_{w,p-1}$}};
            \node (G2) at (1.1,-0.3) {\textcolor{black}{$G_{w,p-1}$}};
        \end{scope}
        
        \begin{scope}[every node/.style={draw, shape = circle, fill = black, minimum size = 0.1cm, inner sep=1pt}]
            % Nodes in complete
            \node (u1) at (-0.75,1.24) {};
            \node (u2) at (-0.25,1.11) {};
            \node (u3) at (0.25,1.11) {};
            \node (u4) at (0.75,1.24) {};
        \end{scope}
    
        \begin{scope}[every node/.style={draw, shape = circle, fill = black, minimum size = 0.1cm, inner sep=1pt}, shift={(-1.1,-0.98-0.3)}]
            % Nodes in G1
            \node (v1) at (-0.75,1.24) {};
            \node (v2) at (-0.25,1.37) {};
            \node (v3) at (0.25,1.37) {};
            \node (v4) at (0.75,1.24) {};
        \end{scope}
    
        \begin{scope}[every node/.style={draw, shape = circle, fill = black, minimum size = 0.1cm, inner sep=1pt}, shift={(1.1,-0.98-0.3)}]
            % Nodes in G2
            \node (w1) at (-0.75,1.24) {};
            \node (w2) at (-0.25,1.37) {};
            \node (w3) at (0.25,1.37) {};
            \node (w4) at (0.75,1.24) {};
        \end{scope}
    
        % Edges
        \begin{scope}[every edge/.style={draw}]
            \foreach \i in {1,2,3,4} {
                \foreach \j in {1,2,3,4} {
                    \path [-] (u\i) edge (v\j);
                    \path [-] (u\i) edge (w\j);
                }
            }
        \end{scope}
    \end{tikzpicture}

    \caption{Constructing $G_{w,p}$}
    \label{fig: constructing Gwp}
\end{figure}

\begin{theorem}\label{thm: construction proof for Gwp}
     For all integers $w\geq0$ and $p\geq p_{w}$, $G_{w,p}$ contains every tree with at most $\gamma(w,p)$ edges.
\end{theorem}
\begin{proof}
    Proceed by induction on $p\geq p_{w}$ with $w$ fixed. If $p=p_{w}$, then \Cref{fact: gamma < delta} says $\gamma(w,p_{w})<\delta(w)$. It follows that any tree on at most $\gamma(w,p_{w})$ edges has at most $\lceil \delta(w)\rceil$ vertices, and is therefore contained in $G_{w,p_{w}}=K_{\lceil \delta(w)\rceil}$. Now assume $p\geq p_{w}+1$ and the claim holds for all lower values of $p$. Let $T$ be a tree on at most $\gamma(w,p)$ edges. \Cref{fact: gamma >= alpha + 3} says $\gamma(w,p)\geq\alpha(w,p)+3$, so we can add vertices and edges to $T$ if necessary (without creating cycles and maintaining connectivity) so that $\alpha(w,p)+1\leq|E(T)|\leq\gamma(w,p)$. By the definition of $p_{w}$, we have $\alpha(w,p)\geq\beta(w)=3(\frac{3}{2})^{w}$. By \Cref{lem: w subtrees induction}, there exists a set $W\subseteq V(T)$ with $|W|\leq w+1$ such that some set $\C$ of $W$-subtrees of $T$ satisfies:
    \begin{align*}
        \alpha(w,p)<\sum_{T'\in\C}|E(T')|\leq\Big(1+\Big(\frac{2}{3}\Big)^{w}\Big)\alpha(w,p)=\gamma(w,p-1)\;,
    \end{align*}
    where the last equality follows from \labelcref{eq: gamma and alpha relations}. Let $F_{1}$ be the union of the $W$-subtrees in $\C$, and let $F_{2}$ be the union of the $W$-subtrees not in $\C$. Since the subtrees in $\C$ are edge-disjoint, the above inequality implies $\alpha(w,p)<|E(F_{1})|\leq\gamma(w,p-1)$. Furthermore, since $\{E(F_{1}),E(F_{2})\}$ is a partition of $E(T)$, 
    \begin{align*}
        |E(F_{2})|=|E(T)|-|E(F_{1})|\leq\gamma(w,p)-\alpha(w,p)&=\Big(2+\Big(\frac{2}{3}\Big)^{w}\Big)\;\alpha(w,p)-\alpha(w,p)\\
        &=\gamma(w,p-1)\;,
    \end{align*}
    where the last two equalities follow from \labelcref{eq: gamma and alpha relations}. For each $i\in\{1,2\}$ let $T_{i}$ be the tree obtained from $F_{i}-W$ by adding edges if necessary to join up components. Observe that $|E(T_{i})|\leq|E(F_{i})|\leq\gamma(w,p-1)$. By the induction hypothesis, $T_{i}$ is contained in $G_{w,p-1}$. Map $T_{i}$ to the $i$-th copy of $G_{w,p-1}$ in $G_{w,p}$. Map $W$ to the $K_{w+1}$ in $G_{w,p}$. Since the $K_{w+1}$ is complete to each copy of $G_{w,p-1}$ in $G_{w,p}$, this shows that $T$ is contained in $G_{w,p}$. The claim holds by induction. 
\end{proof}

%\section{Estimating the number of edges in $G_{w,p}$}
\label{sec: estimating the edges in Gwp}

We now estimate $|V(G_{w,p})|$ and $|E(G_{w,p})|$. Let $m:=p-p_{w}$.

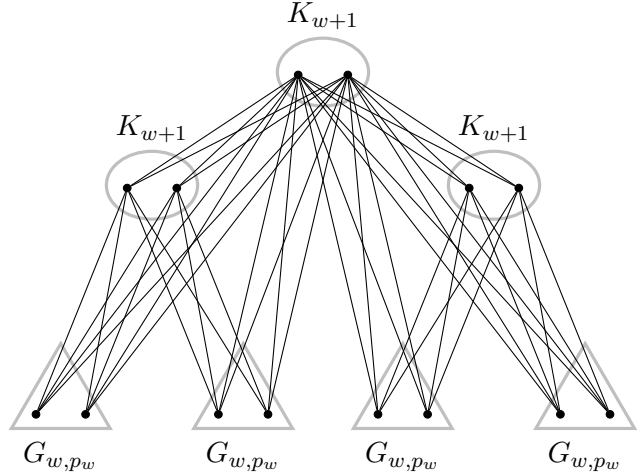
\begin{figure}[!h]
    \centering
    \begin{tikzpicture}[/tikz/xscale=1.5, /tikz/yscale=1.5]
    
        % Kw+1 top nodes
        \begin{scope}[every node/.style={fill=black, shape = circle, minimum size = 0.1cm, inner sep=1pt, draw},shift={(0,3)}]
            \node (u2) at (-0.87/4,-0.5/4) {};
            \node (u3) at (0.87/4,-0.5/4) {};
        \end{scope}
    
        % Kw+1 middle left nodes
        \begin{scope}[every node/.style={fill=black, shape = circle, minimum size = 0.1cm, inner sep=1pt, draw},shift={(-1.5,2)}]
            \node (v2) at (-0.87/4,-0.5/4) {};
            \node (v3) at (0.87/4,-0.5/4) {};
        \end{scope}
    
        % Kw+1 middle right nodes
        \begin{scope}[every node/.style={fill=black, shape = circle, minimum size = 0.1cm, inner sep=1pt, draw},shift={(1.5,2)}]
            \node (w2) at (-0.87/4,-0.5/4) {};
            \node (w3) at (0.87/4,-0.5/4) {};
        \end{scope}
    
        % Gwp left nodes
        \begin{scope}[every node/.style={fill=black, shape = circle, minimum size = 0.1cm, inner sep=1pt, draw},shift={(0.7,0)}]
            \node (x2) at (-0.87/4,-0.5/4) {};
            \node (x3) at (0.87/4,-0.5/4) {};
            \node (x5) at (-0.87/4+1.6,-0.5/4) {};
            \node (x6) at (0.87/4+1.6,-0.5/4) {};
        \end{scope}
    
        % Gwp right nodes
        \begin{scope}[every node/.style={fill=black, shape = circle, minimum size = 0.1cm, inner sep=1pt, draw},shift={(-0.7,0)}]
            \node (y2) at (-0.87/4,-0.5/4) {};
            \node (y3) at (0.87/4,-0.5/4) {};
            \node (y5) at (-0.87/4-1.6,-0.5/4) {};
            \node (y6) at (0.87/4-1.6,-0.5/4) {};
        \end{scope}
    
        % Kw+1 top shapes
        \begin{scope}[every node/.style={ellipse, minimum width=1.2cm, minimum height=0.9cm,very thick,draw, lightgray},shift={(0,-0.1)}]
            \node (K1) at (0,3) {};
        \end{scope}
    
        % Kw+1 top names
        \begin{scope}[shift={(0,-0.2)}]
            \node (nameK1) at (0,3.6) {$K_{w+1}$};
        \end{scope}
    
        % Kw+1 middle shapes
        \begin{scope}[every node/.style={ellipse, minimum width=1.2cm, minimum height=0.9cm, very thick,draw, lightgray},shift={(0,-0.1)}]
            \node (K2) at (-1.5,2) {};
            \node (K3) at (1.5,2) {};
        \end{scope}
    
        % Kw+1 middle names
        \begin{scope}[shift={(0,-0.2)}]
            \node (nameK2) at (-1.5,2.6) {$K_{w+1}$};
            \node (nameK3) at (1.5,2.6) {$K_{w+1}$};
        \end{scope}
        
        % Gwp bottom shapes
        \begin{scope}[every node/.style={regular polygon, regular polygon sides=3, minimum size=1.5cm, very thick,draw, lightgray}]
            \node (G1) at (-2.3,0) {};
            \node (G2) at (-0.7,0) {};
            \node (G3) at (0.7,0) {};
            \node (G4) at (2.3,0) {};
        \end{scope}
    
        % Gwp bottom names
        \begin{scope}[shift={(0,0.1)}]
            \node (nameG1) at (-2.3,-0.6) {$G_{w,p_{w}}$};
            \node (nameG2) at (-0.7,-0.6) {$G_{w,p_{w}}$};
            \node (nameG3) at (0.7,-0.6) {$G_{w,p_{w}}$};
            \node (nameG4) at (2.3,-0.6) {$G_{w,p_{w}}$};
        \end{scope}

        % Edges
        \begin{scope}[every edge/.style={draw}]
            \foreach \i in {2,3} {
                \foreach \j in {2,3} {
                    \path [-] (u\i) edge (v\j);
                    \path [-] (u\i) edge (w\j);
                }
                \foreach \j in {2,3,5,6} {
                    \path [-] (u\i) edge (x\j);
                    \path [-] (v\i) edge (y\j);
                    \path [-] (u\i) edge (y\j);
                    \path [-] (w\i) edge (x\j);
                }
            }
        \end{scope}
    \end{tikzpicture}
    
    \caption{$G_{w,p}$ when $m=2$}
    \label{fig: example of Gwp}
\end{figure}

\begin{lemma}\label{lem: general estimate for vertices in Gwp}
    For all sufficiently large integers $w\geq0$, if $p\geq p_{w}$, then $|V(G_{w,p})|\leq5\cdot2^{p}$.
\end{lemma}
\begin{proof}
    Since $G_{w,p}$ contains $2^{m}-1$ copies of $K_{w+1}$ and $2^{m}$ copies of $G_{w,p_{w}}$ (see \Cref{fig: example of Gwp}),
    \begin{align*}
        |V(G_{w,p})|\leq(2^{m}-1)\cdot|V(K_{w+1})|+2^{m}\cdot|V(G_{w,p_{w}})|<2^{m}(w+1)+2^{m}\lceil\delta(w)\rceil\;.
    \end{align*}
    When $w$ is sufficiently large, $w+1\leq(\frac{3}{2})^{w}$. By \Cref{fact: delta upper bound}, $\lceil\delta(w)\rceil\leq12(\frac{3}{2})^{w}$. Therefore,
    \begin{align*}
        |V(G_{w,p})|\leq2^{m}\cdot\Big(\frac{3}{2}\Big)^{w}+2^{m}\cdot12\Big(\frac{3}{2}\Big)^{w}=2^{m}\cdot13\Big(\frac{3}{2}\Big)^{w}\;.
    \end{align*}
    It follows from \Cref{fact: exponential of pw} that $2^{m}\leq2^{p}\frac{1}{3}(\frac{2}{3})^{w}$. Thus, $|V(G_{w,p})|\leq\frac{13}{3}\cdot2^{p}\leq5\cdot2^{p}$ as desired.
\end{proof}

For the remainder of this section, let $w:=\lceil\log_{3/2}p\rceil$ and consider $p$ large. Note that $p\approx(\frac{3}{2})^{w}$ and $p_{w}\leq w+\log_{3/2}(9)$ by \Cref{fact: pw is linear in w}, thus $p\geq p_{w}$ for all large enough $p$.

\begin{lemma}\label{lem: estimating edges in Gwp}
    $|E(G_{w,p})|\leq5\cdot 2^{p}p\log_{3/2}p+\bigo(2^{p}p)$.
\end{lemma}
\begin{proof}
    Each of the $2^{m}-1$ copies of $K_{w+1}$ in $G_{w,p}$ contains $\binom{w+1}{2}$ edges. Each of the $2^{m}$ copies of $G_{w,p_{w}}$ in $G_{w,p}$ contains $\binom{\lceil\delta(w)\rceil}{2} \leq 12^{2}(\frac{3}{2})^{2w}$ edges by \Cref{fact: delta upper bound}. The contribution of the `downward' edges from each $K_{w+1}$ (see \Cref{fig: example of Gwp}) is the final term in:
    \begin{align*}
        |E(G_{w,p})|\leq(2^{m}-1)\binom{w+1}{2}+2^{m}\cdot12^{2}\Big(\frac{3}{2}\Big)^{2w}+\sum_{i=1}^{m}2^{i}(w+1)|V(G_{w,p-i})|\;.
    \end{align*}
    By \Cref{lem: general estimate for vertices in Gwp} it follows that, for each $i\in\{1,\ldots,m\}$, $2^{i}|V(G_{w,p-i})|\leq5\cdot2^{p}$. Therefore, 
    \begin{align*}
        |E(G_{w,p})|<2^{m}\binom{w+1}{2}+2^{m}\cdot12^{2}\Big(\frac{3}{2}\Big)^{2w}+5\cdot 2^{p}m(w+1)\;.
    \end{align*}
    Since $m\leq p$ and $2^{m}\leq2^{p}\frac{1}{3}(\frac{2}{3})^{w}$ by \Cref{fact: exponential of pw},
    \begin{align*}
        |E(G_{w,p})|\leq2^{p}\frac{1}{3}\Big(\frac{2}{3}\Big)^{w}\binom{w+1}{2}+12^{2}\cdot2^{p}\frac{1}{3}\Big(\frac{3}{2}\Big)^{w}+5\cdot 2^{p}p(w+1)\;.
    \end{align*}
    Since $w=\lceil\log_{3/2}p\rceil$, $5\cdot 2^{p}pw$ is the dominating function in the above inequality. The lower order terms are $\bigo(2^{p}p)$. It follows that
    \begin{align*}
        |E(G_{w,p})|\leq5\cdot 2^{p}pw+\bigo(2^{p}p)\leq5\cdot 2^{p}p\log_{3/2}p+\bigo(2^{p}p)
    \end{align*}
    as claimed.
\end{proof}

For the remainder of this section, let $p:=\lceil\log_{2}(2n)\rceil$ and consider $n$ large.

\begin{theorem}\label{thm: estimating vertices and edges in Gwp in terms of n}
The graph $G_{w,p}$ satisfies $|V(G_{w,p})|\leq20n$ and $|E(G_{w,p})|\leq20n(\log_{2} n)(\log_{3/2}\log_{2} n)+\bigo(n\log n)$. Moreover, $G_{w,p}$ contains every $n$-vertex tree.
\end{theorem}
\begin{proof}
    By \Cref{lem: general estimate for vertices in Gwp}, $|V(G_{w,p})|\leq5\cdot2^{\lceil\log_{2}(2n)\rceil}\leq20n$. By \Cref{lem: estimating edges in Gwp},
    \begin{align*}
        |E(G_{w,p})|&\leq5\cdot2^{\lceil\log_{2}(2n)\rceil}\cdot\lceil\log_{2}(2n)\rceil\cdot\log_{3/2}\lceil\log_{2}(2n)\rceil+\bigo(2^{\lceil\log_{2}(2n)\rceil}\cdot\lceil\log_{2}(2n)\rceil)\\
        &\leq20n(\log_{2} n)(\log_{3/2}\log_{2} n)+\bigo(n\log n)\;.
    \end{align*}
    We now show $\gamma(w,p)\geq n$. We have
    \begin{align*}
        2+\Big(\frac{2}{3}\Big)^{w}>2+\Big(\frac{2}{3}\Big)^{w}-\Big(\frac{2}{3}\Big)^{2w}=\Big(2-\Big(\frac{2}{3}\Big)^{w}\Big)\Big(1+\Big(\frac{2}{3}\Big)^{w}\Big)\;,
    \end{align*}
    implying $\frac{2+(\frac{2}{3})^{w}}{1+(\frac{2}{3})^{w}}\geq2-(\frac{2}{3})^{w}$. Furthermore,
    \begin{align*}
        \Big(1-\frac{1}{2}\Big(\frac{2}{3}\Big)^{w}\Big)^{p}\geq1-\frac{p}{2}\Big(\frac{2}{3}\Big)^{w}\geq\frac{1}{2}\;,
    \end{align*}
    where the first inequality follows from Bernoulli's inequality, and the second inequality follows since $w=\lceil\log_{3/2}p\rceil$. Multiplying both extremes of the above inequality by $2^{p}$ implies $(2-(\frac{2}{3})^{w})^{p}\geq2^{p-1}$. Hence, 
    \begin{align*}
        \gamma(w,p)=\Big(\frac{2+(\frac{2}{3})^{w}}{1+(\frac{2}{3})^{w}}\Big)^{p}\geq\Big(2-\Big(\frac{2}{3}\Big)^{w}\Big)^{p}\geq 2^{p-1}\geq n\;.
    \end{align*}
    Therefore, \Cref{thm: construction proof for Gwp} implies that $G_{w,p}$ contains every $n$-vertex tree.
\end{proof}

\begin{corollary}
    $s(n)\leq\bigo(n(\log n)(\log\log n))$.
\end{corollary}

%%%%%%%%%%%%%%%%%
\section{Treewidth $k$ Graphs}
\label{Treewidthk}

In this section we show that
\begin{align*}
    \Omega(kn\log n)\leq s_{k}(n)\leq\bigo(kn(\log n)(\log\log n))\;.
\end{align*}
For the upper bound, the proof is very similar to the proof for trees, with some minor changes. In particular, the proof for trees is in terms of the number of edges, whereas the proof for treewidth $k$ graphs solely uses the number of vertices. In the proof for trees, the separator in the base case has one vertex, whereas in the $k=1$ case of the treewidth $k$ proof, the analogous separator has two vertices.

The proof of the lower bound follows an argument similar to that used by \citet{GLST23}.

\subsection{Lower bound}

\begin{theorem}\label{thm: lower bound sk(n)}
    $s_k(n)\geq\Omega(kn\log n)$.
\end{theorem}
\begin{proof}
    Let $t(n):=\sqrt{\frac{n}{k}+\frac{1}{4k^{2}}}-\frac{1}{2k}$. Assume $n$ is sufficiently large (relative to $k$) so that $2\leq\frac{1}{2}\sqrt{\frac{n}{k}}+1\leq t(n)\leq\frac{n}{k+1}$. The function $t(n)$ is chosen so $\frac{n}{j}-kj-1$ is a non-negative decreasing function of $j$ in the interval $[1,t(n)]$.
    
    Let $U$ be a graph that contains every graph with $n$ vertices and treewidth at most $k$. We show that $|E(U)|\geq\Omega(kn\log n)$.

    For each $j\in\{1,\dots,\lfloor t(n)\rfloor\}$, let $S_{j}$ be the complete bipartite graph $K_{k,\lfloor\frac{n}{j}\rfloor-k}$; the vertices in the part of size $k$ are called \define{inner} vertices. We write $j\cdot S_{j}$ for the disjoint union of $j$ copies of $S_{j}$. Since $|V(j\cdot S_{j})|\leq n$ and $\tw(j\cdot S_{j})\leq k$, $U$ contains $j\cdot S_{j}$. Let $H_{j}$ be a copy of $j\cdot S_{j}$ in $U$. Note that $H_{j}$ has $kj$ inner vertices.
    
    Mark some of the edges of $U$ as follows. Let $X_{1}$ be the set of inner vertices of $H_{1}$. Mark every edge of $H_{1}$ with an end vertex in $X_{1}$; there are $k(\lfloor\frac{n}{1}\rfloor-k)$ such edges. Apply the following algorithm: for $j=2,\ldots,\lfloor t(n)\rfloor$, let $X_{j}$ be a set of $k$ inner vertices of $H_{j}$ such that $(X_{1}\cup\dots\cup X_{j-1})\cap X_{j}=\varnothing$. Mark every edge of $H_{j}$ with an end vertex in $X_{j}$.

    At the end of step $j$, for each $x\in X_{j}$, the number of new marked edges in $U$ incident to $x$ is at least $\lfloor\frac{n}{j}\rfloor-k$ minus the number of vertices in $X_{1}\cup\dots\cup X_{j-1}$ that are in the same component of $H_{j}$ as $x$. Since $|X_{1}\cup\dots\cup X_{j-1}|=k(j-1)$ and $|X_{j}|=k$, it follows that at the end of step $j$, the number of new marked edges in $U$ is at least $k(\lfloor\frac{n}{j}\rfloor-k)-k^{2}(j-1)\geq k(\frac{n}{j}-kj-1)$. Thus, at the end of the algorithm, the total number of marked edges in $U$ is at least
    \begin{align*}
        \sum_{j=1}^{\lfloor t(n)\rfloor}k\Big(\frac{n}{j}-kj-1\Big)\geq k\int_{1}^{\frac{1}{2}\sqrt{\frac{n}{k}}}\Big(\frac{n}{j}-kj-1\Big)\;dj
        &=k\bigg[n\log j-\frac{k}{2}j^{2}-j\bigg]_{j=1}^{\frac{1}{2}\sqrt{\frac{n}{k}}}\\
        &\geq kn\log\Big(\frac{1}{2}\sqrt{\frac{n}{k}}\Big)-k\bigg[\frac{k}{2}j^{2}+j\bigg]_{j=1}^{\frac{1}{2}\sqrt{\frac{n}{k}}}\\
        &\geq\frac{1}{2}kn\log n-\bigo(kn\log k)-\bigo(kn)\\
        &\geq\Omega(kn\log n)\;.\qedhere
    \end{align*}
\end{proof}

\subsection{Normal tree-decompositions}

A tree-decomposition with width $k$ is \define{normal} if every bag has size $k+1$, and the intersection of any two adjacent bags has size $k$. Every graph has a tree-decomposition of minimum width that is normal \cite{HW17}. In this section we prove \Cref{lem: z subtrees}; a result relating to normal tree-decompositions which is used in \Cref{sec: w components of graphs}. We first introduce some new notation.

Given a tree-decomposition $(B_{x}:x\in V(T))$ of a graph $G$ and a vertex $z$ of $T$, for any subtree $T'$ of $T$ define \define{$G(T',z):=(\bigcup_{x\in V(T')}B_{x})\setminus B_{z}$}. Then for any real number $\alpha\geq1$, $T'$ is said to be \define{$(\alpha,z)$-light} if $|G(T',z)|\leq\alpha$, and \define{$(\alpha,z)$-heavy} if $|G(T',z)|>\alpha$.

Recall that a $z$-subtree of $T$ is a maximal subtree of $T$ in which $z$ only appears as a leaf. Note that if $\mathcal{F}$ is the set of all $z$-subtrees of $T$, then $\{G(F,z):F\in\mathcal{F}\}$ partitions $V(G)\setminus B_{z}$.

\begin{lemma}\label{lem: z subtrees}
    For every integer $k\geq1$, for every real number $\alpha\geq1$, for every graph $G$ with $|V(G)|>\alpha+k+1$ and treewidth $k$, for every normal tree-decomposition $(B_{x}:x\in V(T))$ of $G$ with width $k$, there exists a vertex $z\in V(T)$ such that for some set $\C$ of $z$-subtrees of $T$,
    \begin{align*}
        \alpha<\sum_{T'\in\C}|G(T',z)|\leq2\alpha\;.
    \end{align*}
\end{lemma}
\begin{proof}
    Let $y$ be a leaf in $T$. Since $|G(T,y)|=|V(G)|-(k+1)>\alpha$, $T$ is $(\alpha,y)$-heavy. We may assume $T$ is $(2\alpha,y)$-heavy, otherwise $\{T\}$ is a desired set of $y$-subtrees of $T$. Let $z$ be the neighbour of $y$ in $T$ and let $v\in B_{y}\setminus B_{z}$. Note that $v$ does not appear in a bag of $T-y$. Therefore,
    \begin{align*}
        \Big|\Big(\bigcup_{x\in V(T)}B_{x}\Big)\setminus\{v\}\Big|=|G(T,y)|+k>2\alpha+k\geq\alpha+k+1\;.
    \end{align*}
    Suppose that every $z$-subtree of $T-y$ is $(\alpha,z)$-light. Then let $\{T_{1},\ldots,T_{j}\}$ be a minimal set of $z$-subtrees of $T-y$ such that $|G(T_{1},z)|+\cdots+|G(T_{j},z)|>\alpha$ (this is well-defined since the union of all the $z$-subtrees of $T-y$ is $(\alpha,z)$-heavy). If $|G(T_{1},z)|+\cdots+|G(T_{j},z)|>2\alpha$, then the minimality of $\{T_{1},\ldots,T_{j}\}$ implies
    \begin{align*}
        |G(T_{j},z)|=|G(T_{1},z)|+\cdots+|G(T_{j},z)|-(|G(T_{1},z)|+\cdots+|G(T_{j-1},z)|)>2\alpha-\alpha=\alpha\;,
    \end{align*}
    contradicting that $T_{j}$ is $(\alpha,z)$-light. Thus, $\{T_{1},\ldots,T_{j}\}$ is a desired set of $z$-subtrees of $T$.

    Now assume some $z$-subtree of $T-y$ is $(\alpha,z)$-heavy. Let $P=(v_{0},v_{1},\dots,v_{t})$ be a maximal path in $T$ such that $v_{0}=y$, $v_{1}=z$ and for each $i\geq1$ there exists an $(\alpha,v_{i})$-heavy $v_{i}$-subtree of $T-v_{i-1}$ that contains $v_{i}$. Let $T'$ be an $(\alpha,v_{t})$-heavy $v_{t}$-subtree of $T-v_{t-1}$ that contains $v_{t}$. We may assume $T'$ is $(2\alpha,v_{t})$-heavy, otherwise $\{T'\}$ is a desired set of $v_{t}$-subtrees of $T$. Let $v_{t+1}$ be the neighbour of $v_{t}$ in $T'$, and let $u\in B_{v_{t}}\setminus B_{v_{t+1}}$. Note that $u$ does not appear in a bag of $T'-v_{t}$. Therefore, 
    \begin{align*}
        \Big|\Big(\bigcup_{x\in V(T')}B_{x}\Big)\setminus\{u\}\Big|=|G(T',v_{t})|+k>2\alpha+k\geq\alpha+k+1\;.
    \end{align*}
    By the maximality of $P$, every $v_{t+1}$-subtree of $T'-v_{t}$ is $(\alpha,v_{t+1})$-light. By the same argument given above, a minimal set of $v_{t+1}$-subtrees of $T'-v_{t}$ whose union is $(\alpha,v_{t+1})$-heavy will suffice. This proves the lemma.
\end{proof}

\subsection{$W$-components of graphs}\label{sec: w components of graphs}

In this section, we prove a result analogous to \Cref{lem: w subtrees induction} but for graphs with bounded treewidth. To do this, we first introduce a modified notion of $W$-subtrees from \Cref{sec: w subtrees of forests} that generalises to arbitrary graphs. For a graph $G$ and a set $W\subseteq V(G)$, a \define{$W$-component} of $G$ is a connected component of $G-W$.

\begin{lemma}\label{lem: w components induction}
    For all integers $k\geq1$ and $w\geq0$, for every real number $\alpha\geq(2k+4)(\frac{3}{2})^{w}$, for every graph $G$ with $|V(G)|>\alpha+k+1$ and $\tw(G)\leq k$, there exists $W\subseteq V(G)$ with $|W|\leq(k+1)(w+1)$ such that for some set $\mathcal{S}$ of $W$-components of $G$,
    \begin{align*}
        \alpha<\sum_{C\in\mathcal{S}}|V(C)|\leq\Big(1+\Big(\frac{2}{3}\Big)^{w}\Big)\alpha\;.
    \end{align*}
\end{lemma}

The proof of this result requires the following lemma analogous to \Cref{lem: w subtrees base case}.

\begin{lemma}\label{lem: w components basis}
    For every integer $k\geq1$, for every real number $\alpha\geq1$, for every graph $G$ with $|V(G)|>\alpha+k+1$ and $\tw(G)\leq k$, there exists $W\subseteq V(G)$ with $|W|=k+1$ such that for some set $\mathcal{S}$ of $W$-components of $G$,
    \begin{align*}
        \alpha<\sum_{C\in\mathcal{S}}|V(C)|\leq2\alpha\;.
    \end{align*}
\end{lemma}
\begin{proof}
    Let $G'$ be the graph obtained from $G$ by adding edges if necessary such that $\tw(G')=k$. Let $(B_{x}:x\in V(T))$ be a normal tree-decomposition of $G'$ with width $k$. By \Cref{lem: z subtrees}, there is a vertex $z\in V(T)$ such that for some set $\C$ of $z$-subtrees of $T$, $\alpha<\sum_{T'\in\C}|G(T',z)|\leq2\alpha$. Let $W:=B_{z}$ and let $\mathcal{S}$ be the set of components of $G[\bigcup_{T'\in\C}G(T',z)]$. The proof is completed by observing that $|W|=k+1$ and $\alpha<\sum_{C\in\mathcal{S}}|V(C)|\leq2\alpha$.
\end{proof}

\begin{proof}[Proof of \Cref{lem: w components induction}]
    Proceed by induction on $w$ with $k$ fixed. The $w=0$ case is a direct application of \Cref{lem: w components basis}. Assume the result holds for some $w\geq0$. Let $\alpha\geq(2k+4)(\frac{3}{2})^{w+1}$ and let $G$ be a graph with $|V(G)|>\alpha+k+1$ and $\tw(G)\leq k$. Since $\alpha\geq(2k+4)(\frac{3}{2})^{w}$, by the induction hypothesis there exists $W\subseteq V(G)$ with $|W|\leq(k+1)(w+1)$ such that for some set $\mathcal{S}(W)$ of $W$-components of $G$, $\alpha<\sum_{C\in\mathcal{S}(W)}|V(C)|\leq(1+(\frac{2}{3})^{w})\alpha$. We are done if $\sum_{C\in\mathcal{S}(W)}|V(C)|\leq(1+(\frac{2}{3})^{w+1})\alpha$, so assume the contrary. Then
    \begin{align}\label{lem: w components induction: ineq 1}
        \Big(1+\Big(\frac{2}{3}\Big)^{w+1}\Big)\alpha<\sum_{C\in\mathcal{S}(W)}|V(C)|\leq\Big(1+\Big(\frac{2}{3}\Big)^{w}\Big)\alpha\;.
    \end{align}
    Define the graph $H:=\bigcup_{C\in\mathcal{S}(W)}C$ and let $\beta:=(\frac{2}{3})^{w+1}\frac{\alpha}{2}-k-1\geq1$. Then $\tw(H)\leq\tw(G)\leq k$. Since the graphs in $\mathcal{S}(W)$ are vertex-disjoint,
    \begin{align*}
        |V(H)|=\sum_{C\in\mathcal{S}(W)}|V(C)|>\alpha+\Big(\frac{2}{3}\Big)^{w+1}\alpha=\alpha+2(\beta+k+1)>\beta+k+1\;.
    \end{align*}
    Since $\beta\geq1$, \Cref{lem: w components basis} implies that there exists $X\subseteq V(H)$ with $|X|=k+1$ such that for some set $\mathcal{S}(X)$ of $X$-components of $H$,
    \begin{align}\label{lem: w components induction: ineq 2}
        \Big(\frac{2}{3}\Big)^{w+1}\frac{\alpha}{2}-k-1=\beta<\sum_{C\in\mathcal{S}(X)}|V(C)|\leq2\beta=\Big(\frac{2}{3}\Big)^{w+1}\alpha-2k-2\;.
    \end{align}
    Then \labelcref{lem: w components induction: ineq 1,lem: w components induction: ineq 2} together imply
    \begin{align}
        \alpha<\alpha+(k+1)&=\Big(1+\Big(\frac{2}{3}\Big)^{w+1}\Big)\alpha-(k+1)-\Big(\Big(\frac{2}{3}\Big)^{w+1}\alpha-2k-2\Big)\nonumber\\
        &<\sum_{C\in\mathcal{S}(W)}|V(C)|-(k+1)-\sum_{C\in\mathcal{S}(X)}|V(C)|\label{lem: w components induction: ineq 3}\\
        &\leq\Big(1+\Big(\frac{2}{3}\Big)^{w}\Big)\alpha-(k+1)-\Big(\Big(\frac{2}{3}\Big)^{w+1}\frac{\alpha}{2}-k-1\Big)\nonumber\\
        &=\alpha+\Big(\frac{2}{3}\Big)^{w}\alpha-\Big(\frac{2}{3}\Big)^{w+1}\frac{\alpha}{2}\nonumber\\
        &=\Big(1+\Big(\frac{2}{3}\Big)^{w+1}\Big)\alpha\;.\nonumber
    \end{align}
    Let $\overline{\mathcal{S}(X)}$ be the set of all $X$-components of $H$ not in $\mathcal{S}(X)$. Since $\{V(C):C\in\mathcal{S}(W)\}$ and $\{X\}\cup\{V(C):C\in\mathcal{S}(X)\}\cup\{V(C):C\in\overline{\mathcal{S}(X)}\}$ are both partitions of $V(H)$,
    \begin{align*}
        |V(H)|=\sum_{C\in\mathcal{S}(W)}|V(C)|=(k+1)+\sum_{C\in\mathcal{S}(X)}|V(C)|+\sum_{C\in\overline{\mathcal{S}(X)}}|V(C)|\;.
    \end{align*}
    Therefore, \labelcref{lem: w components induction: ineq 3} implies
    \begin{align*}
        \alpha<\sum_{C\in\overline{\mathcal{S}(X)}}|V(C)|\leq\Big(1+\Big(\frac{2}{3}\Big)^{w+1}\Big)\alpha\;.
    \end{align*}
    Thus, $\overline{\mathcal{S}(X)}$ is a desired set of $(W\cup X)$-components of $G$. The result follows by induction.
\end{proof}

\subsection{Useful functions}

This section proves a number of elementary facts that are direct extensions of the analogous observations in \Cref{sec: useful functions}. For integers $w,p\geq0$ define $\alpha(w,p)$ and $\gamma(w,p)$ as in \Cref{sec: useful functions}. For integers $w,p\geq0$ and $k\geq1$, define
\begin{align*}
    \beta_{k}(w):=(2k+4)\Big(\frac{3}{2}\Big)^{w}\qquad\text{and}\qquad\delta_{k}(w):=\Big(2+\Big(\frac{2}{3}\Big)^{w}\Big)\beta_{k}(w)\;.
\end{align*}
While the $k=1$ case corresponds to trees, $\beta_{1}(w)$ and $\delta_{1}(w)$ are slightly larger than $\beta(w)$ and $\delta(w)$ respectively. This is because of the $(k+1)(w+1)$ term in \Cref{lem: w components induction}.

For integers $w\geq0$ and $k\geq1$, observe that for all sufficiently large integers $p$, $\alpha(w,p)\geq\frac{1}{2}(\frac{3}{2})^{p-1}\geq\beta_{k}(w)$. Define $p_{w,k}$ to be the smallest integer such that $\alpha(w,p_{w,k}+1)\geq\beta_{k}(w)$. Since $\alpha(w,p)$ is an increasing function of $p$, if $p\geq p_{w,k}+1$ then $\alpha(w,p)\geq\beta_{k}(w)$, and if $p\leq p_{w,k}$ then $\alpha(w,p)<\beta_{k}(w)$.

Let $k$ be a fixed positive integer for the remainder of this section.

\begin{fact}\label{fact: pwk is linear in w}
     For every integer $w\geq0$, $\log_{2}(\frac{3}{2})\cdot w+\log_{2}(2k+4)\leq p_{w,k}\leq w+\log_{3/2}(6k+12)$.
\end{fact}
\begin{proof}
    Observe that $p_{w,k}=\lceil x\rceil$ where $\alpha(w,x+1)=\beta_{k}(w)$. Rearranging $\alpha(w,x+1)=\beta_{k}(w)$ for $x$ gives 
    \begin{align*}
        x=\frac{\log\Big((1+(\frac{2}{3})^{w})\cdot(2k+4)(\frac{3}{2})^{w}\Big)}{\log\Big(\frac{2+(\frac{2}{3})^{w}}{1+(\frac{2}{3})^{w}}\Big)}\;.
    \end{align*}
    Let $a=\log((1+(\frac{2}{3})^{w})\cdot(2k+4)(\frac{3}{2})^{w})$ and $b=\log\Big(\frac{2+(\frac{2}{3})^{w}}{1+(\frac{2}{3})^{w}}\Big)$. Then, $a>\log((2k+4)(\frac{3}{2})^{w})$ and $b\leq\log(2)$, implying $\log_{2}(\frac{3}{2})\cdot w+\log_{2}(2k+4)\leq \frac{a}{b}\leq p_{w,k}$. Furthermore,
    \begin{align*}
        a&\leq\log\Big(\Big(1+\Big(\frac{2}{3}\Big)^{0}\Big)\cdot(2k+4)\Big(\frac{3}{2}\Big)^{w}\Big)=\log\Big((4k+8)\Big(\frac{3}{2}\Big)^{w}\Big)\quad\text{and}\\
        b&\geq\log\Big(\frac{2+(\frac{2}{3})^{0}}{1+(\frac{2}{3})^{0}}\Big)=\log\Big(\frac{3}{2}\Big)\;,
    \end{align*}
    implying $\frac{a}{b}\leq w+\log_{3/2}(4k+8)$. Thus, $p_{w,k}\leq w+\log_{3/2}(4k+8)+1=w+\log_{3/2}(6k+12)$.
\end{proof}

\begin{fact}\label{fact: gamma >= alpha + k + 4}
    For all integers $w\geq0$ and $p\geq p_{w,k}+1$, $\gamma(w,p)\geq\alpha(w,p)+k+4$.
\end{fact}
\begin{proof}
    By the definition of $p_{w,k}$, $\alpha(w,p)\geq\beta_{k}(w)$ for all $w\geq0$ and $p\geq p_{w,k}+1$. By \labelcref{eq: gamma and alpha relations},
    \begin{align*}
        \gamma(w,p)-\alpha(w,p)=\Big(2+\Big(\frac{2}{3}\Big)^{w}\Big)\alpha(w,p)-\alpha(w,p)&=\Big(1+\Big(\frac{2}{3}\Big)^{w}\Big)\alpha(w,p)\\
        &\geq\Big(1+\Big(\frac{2}{3}\Big)^{w}\Big)\beta_{k}(w)\\
        &>\beta_{k}(w)\;.
    \end{align*}
    Since $\beta_{k}(w)\geq2k+4$, $\gamma(w,p)-\alpha(w,p)\geq k+4$.
\end{proof}

\begin{fact}\label{fact: gamma < deltak}
     For every integer $w\geq0$, $\gamma(w,p_{w,k})<\delta_{k}(w)$.
\end{fact}
\begin{proof}
    By the definition of $p_{w,k}$, $\alpha(w,p_{w,k})<\beta_{k}(w)$. By \labelcref{eq: gamma and alpha relations},
    \[\gamma(w,p_{w,k})=\Big(2+\Big(\frac{2}{3}\Big)^{w}\Big)\alpha(w,p_{w,k})<\Big(2+\Big(\frac{2}{3}\Big)^{w}\Big)\beta_{k}(w)=\delta_{k}(w)\;.\qedhere\]
\end{proof}

\begin{fact}\label{fact: deltak upper bound}
     For every integer $w\geq0$, $\lceil \delta_{k}(w)\rceil\leq(8k+16)(\frac{3}{2})^{w}$.
\end{fact}
\begin{proof}
    By the definition of $\delta_{k}(w)$ and $\beta_{k}(w)$,
    \[\lceil \delta_{k}(w)\rceil=\Big\lceil\Big(2+\Big(\frac{2}{3}\Big)^{w}\Big)\beta_{k}(w)\Big\rceil\leq\lceil3\beta_{k}(w)\rceil\leq4\beta_{k}(w)=(8k+16)\Big(\frac{3}{2}\Big)^{w}\;.\qedhere\]
\end{proof}

\begin{fact}\label{fact: exponential of pwk}
     For every integer $w\geq0$, $2^{-p_{w,k}}\leq\frac{1}{2k+4}(\frac{2}{3})^{w}$.
\end{fact}
\begin{proof}
    By \Cref{fact: pwk is linear in w}, $\log_{2}(\frac{3}{2})\cdot w+\log_{2}(2k+4)\leq p_{w,k}$. Therefore,
    \[2^{-p_{w,k}}\leq2^{-\log_{2}(\frac{3}{2})\cdot w-\log_{2}(2k+4)}=\frac{1}{2k+4}\Big(\frac{2}{3}\Big)^{w}\;.\qedhere\]
\end{proof}

\subsection{The universal graph $G^{k}_{w,p}$}

For integers $w\geq0$, $k\geq1$ and $p\geq p_{w,k}$, define the graph $G^{k}_{w,p}$ recursively. Let $G^{k}_{w,p_{w,k}}:=K_{\lceil \delta_{k}(w)\rceil}$ and for $p\geq p_{w,k}+1$, let $G^{k}_{w,p}$ be the graph formed by adding all the possible edges between a complete graph $K_{(k+1)(w+1)}$ and two disjoint copies of $G^{k}_{w,p-1}$, as shown in \Cref{fig: constructing Gkwp}.

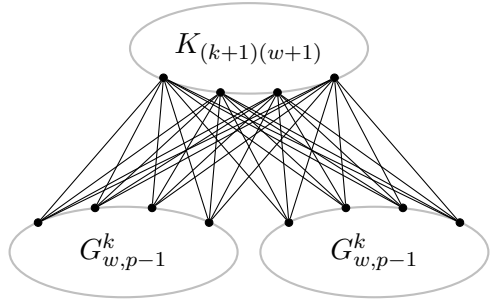
\begin{figure}[!h]
    \centering
    \begin{tikzpicture}[/tikz/xscale=1.5, /tikz/yscale=1.5]

        % Blobs
        \begin{scope}[every node/.style={ellipse, minimum width=3cm, minimum height=1.2cm, thick,draw, lightgray}]
            \node (complete) at (0,1.5) {\textcolor{black}{$K_{(k+1)(w+1)}$}};
            \node (G1) at (-1.1,-0.3) {\textcolor{black}{$G^{k}_{w,p-1}$}};
            \node (G2) at (1.1,-0.3) {\textcolor{black}{$G^{k}_{w,p-1}$}};
        \end{scope}
        
        \begin{scope}[every node/.style={draw, shape = circle, fill = black, minimum size = 0.1cm, inner sep=1pt}]
            % Nodes in complete
            \node (u1) at (-0.75,1.24) {};
            \node (u2) at (-0.25,1.11) {};
            \node (u3) at (0.25,1.11) {};
            \node (u4) at (0.75,1.24) {};
        \end{scope}
    
        \begin{scope}[every node/.style={draw, shape = circle, fill = black, minimum size = 0.1cm, inner sep=1pt}, shift={(-1.1,-0.98-0.3)}]
            % Nodes in G1
            \node (v1) at (-0.75,1.24) {};
            \node (v2) at (-0.25,1.37) {};
            \node (v3) at (0.25,1.37) {};
            \node (v4) at (0.75,1.24) {};
        \end{scope}
    
        \begin{scope}[every node/.style={draw, shape = circle, fill = black, minimum size = 0.1cm, inner sep=1pt}, shift={(1.1,-0.98-0.3)}]
            % Nodes in G2
            \node (w1) at (-0.75,1.24) {};
            \node (w2) at (-0.25,1.37) {};
            \node (w3) at (0.25,1.37) {};
            \node (w4) at (0.75,1.24) {};
        \end{scope}
    
        % Edges
        \begin{scope}[every edge/.style={draw}]
            \foreach \i in {1,2,3,4} {
                \foreach \j in {1,2,3,4} {
                    \path [-] (u\i) edge (v\j);
                    \path [-] (u\i) edge (w\j);
                }
            }
        \end{scope}
    \end{tikzpicture}

    \caption{Constructing $G^{k}_{w,p}$}
    \label{fig: constructing Gkwp}
\end{figure}

\begin{theorem}\label{thm: construction proof for Gkwp}
     For all integers $w\geq0$, $k\geq1$ and $p\geq p_{w,k}$, $G^{k}_{w,p}$ contains every graph with at most $\gamma(w,p)$ vertices and treewidth at most $k$.
\end{theorem}
\begin{proof}
    Proceed by induction on $p\geq p_{w,k}$ with $w$ and $k$ fixed. If $p=p_{w,k}$, then \Cref{fact: gamma < deltak} says $\gamma(w,p_{w,k})<\delta_{k}(w)$. Every graph on at most $\gamma(w,p_{w,k})$ vertices is contained in $G^{k}_{w,p_{w,k}}=K_{\lceil \delta_{k}(w)\rceil}$, so the base case holds. Now assume $p\geq p_{w,k}+1$ and the claim holds for all lower values of $p$. Let $G$ be a graph on at most $\gamma(w,p)$ vertices with $\tw(G)\leq k$. \Cref{fact: gamma >= alpha + k + 4} says $\gamma(w,p)\geq\alpha(w,p)+k+4$, so we can add vertices and edges to $G$ if necessary (without increasing the treewidth) so that $\alpha(w,p)+k+1<|V(G)|\leq\gamma(w,p)$. By the definition of $p_{w,k}$, we have $\alpha(w,p)\geq\beta_{k}(w)=(2k+4)(\frac{3}{2})^{w}$. By \Cref{lem: w components induction}, there exists a set $W\subseteq V(G)$ with $|W|\leq(k+1)(w+1)$ such that some set $\mathcal{S}$ of $W$-components of $G$ satisfies:
    \begin{align*}
        \alpha(w,p)<\sum_{C\in\mathcal{S}}|V(C)|\leq\Big(1+\Big(\frac{2}{3}\Big)^{w}\Big)\alpha(w,p)=\gamma(w,p-1)\;,
    \end{align*}
    where the last equality follows from \labelcref{eq: gamma and alpha relations}. Let $H_{1}$ be the union of the $W$-components in $\mathcal{S}$, and let $H_{2}$ be the union of the $W$-components not in $\mathcal{S}$. Since the graphs in $\mathcal{S}$ are vertex-disjoint, the above inequality implies $\alpha(w,p)<|V(H_{1})|\leq\gamma(w,p-1)$. Furthermore, since $\{W,V(H_{1}),V(H_{2})\}$ is a partition of $V(G)$, 
    \begin{align*}
        |V(H_{2})|=|V(G)|-|W|-|V(H_{1})|\leq|V(G)|-|V(H_{1})|&\leq\gamma(w,p)-\alpha(w,p)\\
        &=\Big(2+\Big(\frac{2}{3}\Big)^{w}\Big)\alpha(w,p)-\alpha(w,p)\\
        &=\gamma(w,p-1)\;,
    \end{align*}
    where the last two equalities follow from \labelcref{eq: gamma and alpha relations}. Note that for each $i\in\{1,2\}$, $\tw(H_{i})\leq\tw(G)\leq k$. By the induction hypothesis, $H_{i}$ is contained in $G^{k}_{w,p-1}$. Map $H_{i}$ to the $i$-th copy of $G^{k}_{w,p-1}$ in $G^{k}_{w,p}$. Map $W$ to the $K_{(k+1)(w+1)}$ in $G^{k}_{w,p}$. Since the $K_{(k+1)(w+1)}$ is complete to each copy of $G^{k}_{w,p-1}$ in $G^{k}_{w,p}$, this shows that $G$ is contained in $G^{k}_{w,p}$. The claim holds by induction. 
\end{proof}

%\section{Estimating the number of edges in $G^{k}_{w,p}$}

We now estimate $|V(G^{k}_{w,p})|$ and $|E(G^{k}_{w,p})|$. Let $m:=p-p_{w,k}$. 

\begin{lemma}\label{lem: general estimate for vertices in Gkwp}
    For every integer $k\geq1$, for every integer $w$ sufficiently large relative to $k$, if $p\geq p_{w,k}$, then $|V(G^{k}_{w,p})|\leq5\cdot2^{p}$.
\end{lemma}
\begin{proof}
    Since $G^{k}_{w,p}$ contains $2^{m}-1$ copies of $K_{(k+1)(w+1)}$ and $2^{m}$ copies of $G^{k}_{w,p_{w,k}}$ (see \Cref{fig: example of Gkwp}),
    \begin{align*}
        |V(G^{k}_{w,p})|\leq(2^{m}-1)\cdot|V(K_{(k+1)(w+1)})|+2^{m}\cdot|V(G^{k}_{w,p_{w,k}})|<2^{m}(k+1)(w+1)+2^{m}\lceil\delta_{k}(w)\rceil\;.
    \end{align*}
    \Cref{fact: deltak upper bound} says that $\lceil\delta_{k}(w)\rceil\leq(8k+16)(\frac{3}{2})^{w}$. Furthermore, when $w$ is sufficiently large relative to $k$, $(k+1)(w+1)\leq(\frac{3}{2})^{w}$. Therefore,
    \begin{align*}
        |V(G^{k}_{w,p})|\leq2^{m}\cdot\Big(\frac{3}{2}\Big)^{w}+2^{m}\cdot(8k+16)\Big(\frac{3}{2}\Big)^{w}=2^{m}\cdot(8k+17)\Big(\frac{3}{2}\Big)^{w}\;.
    \end{align*}
    By \Cref{fact: exponential of pwk}, $2^{m}\leq2^{p}\frac{1}{2k+4}(\frac{2}{3})^{w}$. Thus, $|V(G^{k}_{w,p})|\leq\frac{8k+17}{2k+4}\cdot2^{p}<5\cdot2^{p}$ as desired.
\end{proof}

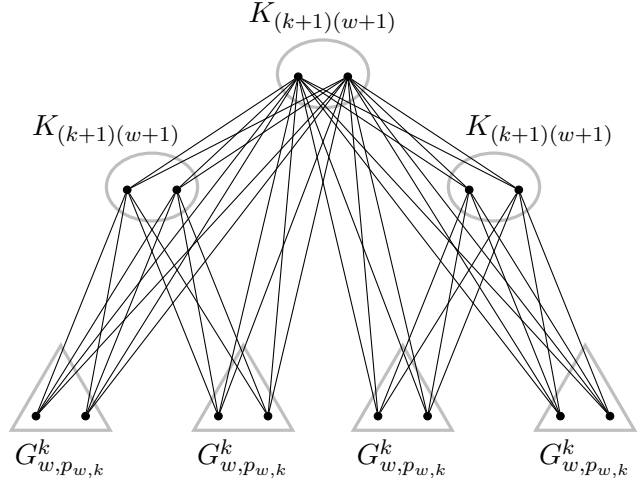
\begin{figure}[!h]
    \centering
    \begin{tikzpicture}[/tikz/xscale=1.5, /tikz/yscale=1.5]
    
        % Kw+1 top nodes
        \begin{scope}[every node/.style={fill=black, shape = circle, minimum size = 0.1cm, inner sep=1pt, draw},shift={(0,3)}]
            \node (u2) at (-0.87/4,-0.5/4) {};
            \node (u3) at (0.87/4,-0.5/4) {};
        \end{scope}
    
        % Kw+1 middle left nodes
        \begin{scope}[every node/.style={fill=black, shape = circle, minimum size = 0.1cm, inner sep=1pt, draw},shift={(-1.5,2)}]
            \node (v2) at (-0.87/4,-0.5/4) {};
            \node (v3) at (0.87/4,-0.5/4) {};
        \end{scope}
    
        % Kw+1 middle right nodes
        \begin{scope}[every node/.style={fill=black, shape = circle, minimum size = 0.1cm, inner sep=1pt, draw},shift={(1.5,2)}]
            \node (w2) at (-0.87/4,-0.5/4) {};
            \node (w3) at (0.87/4,-0.5/4) {};
        \end{scope}
    
        % Gwp left nodes
        \begin{scope}[every node/.style={fill=black, shape = circle, minimum size = 0.1cm, inner sep=1pt, draw},shift={(0.7,0)}]
            \node (x2) at (-0.87/4,-0.5/4) {};
            \node (x3) at (0.87/4,-0.5/4) {};
            \node (x5) at (-0.87/4+1.6,-0.5/4) {};
            \node (x6) at (0.87/4+1.6,-0.5/4) {};
        \end{scope}
    
        % Gwp right nodes
        \begin{scope}[every node/.style={fill=black, shape = circle, minimum size = 0.1cm, inner sep=1pt, draw},shift={(-0.7,0)}]
            \node (y2) at (-0.87/4,-0.5/4) {};
            \node (y3) at (0.87/4,-0.5/4) {};
            \node (y5) at (-0.87/4-1.6,-0.5/4) {};
            \node (y6) at (0.87/4-1.6,-0.5/4) {};
        \end{scope}
    
        % Kw+1 top shapes
        \begin{scope}[every node/.style={ellipse, minimum width=1.2cm, minimum height=0.9cm,very thick,draw, lightgray},shift={(0,-0.1)}]
            \node (K1) at (0,3) {};
        \end{scope}
    
        % Kw+1 top names
        \begin{scope}[shift={(0,-0.2)}]
            \node (nameK1) at (0,3.6) {$K_{(k+1)(w+1)}$};
        \end{scope}
    
        % Kw+1 middle shapes
        \begin{scope}[every node/.style={ellipse, minimum width=1.2cm, minimum height=0.9cm, very thick,draw, lightgray},shift={(0,-0.1)}]
            \node (K2) at (-1.5,2) {};
            \node (K3) at (1.5,2) {};
        \end{scope}
    
        % Kw+1 middle names
        \begin{scope}[shift={(0,-0.2)}]
            \node (nameK2) at (-1.9,2.6) {$K_{(k+1)(w+1)}$};
            \node (nameK3) at (1.9,2.6) {$K_{(k+1)(w+1)}$};
        \end{scope}
        
        % Gwp bottom shapes
        \begin{scope}[every node/.style={regular polygon, regular polygon sides=3, minimum size=1.5cm, very thick,draw, lightgray}]
            \node (G1) at (-2.3,0) {};
            \node (G2) at (-0.7,0) {};
            \node (G3) at (0.7,0) {};
            \node (G4) at (2.3,0) {};
        \end{scope}
    
        % Gwp bottom names
        \begin{scope}[shift={(0,0.1)}]
            \node (nameG1) at (-2.3,-0.6) {$G^{k}_{w,p_{w,k}}$};
            \node (nameG2) at (-0.7,-0.6) {$G^{k}_{w,p_{w,k}}$};
            \node (nameG3) at (0.7,-0.6) {$G^{k}_{w,p_{w,k}}$};
            \node (nameG4) at (2.3,-0.6) {$G^{k}_{w,p_{w,k}}$};
        \end{scope}

        % Edges
        \begin{scope}[every edge/.style={draw}]
            \foreach \i in {2,3} {
                \foreach \j in {2,3} {
                    \path [-] (u\i) edge (v\j);
                    \path [-] (u\i) edge (w\j);
                }
                \foreach \j in {2,3,5,6} {
                    \path [-] (u\i) edge (x\j);
                    \path [-] (v\i) edge (y\j);
                    \path [-] (u\i) edge (y\j);
                    \path [-] (w\i) edge (x\j);
                }
            }
        \end{scope}
    \end{tikzpicture}
    
    \caption{$G^{k}_{w,p}$ when $m=2$}
    \label{fig: example of Gkwp}
\end{figure}

For the remainder of this section, let $w:=\lceil\log_{3/2}p\rceil$ and consider $p$ large. Note that $p\approx(\frac{3}{2})^{w}$ and $p_{w,k}\leq w+\log_{3/2}(6k+12)$ by \Cref{fact: pwk is linear in w}, thus $p\geq p_{w,k}$ for all large enough $p$.

\begin{lemma}\label{lem: estimating edges in Gkwp}
    $|E(G^{k}_{w,p})|\leq5k2^{p}p\log_{3/2} p+\bigo(k2^{p}p)$.
\end{lemma}
\begin{proof}
    Each of the $2^{m}-1$ copies of $K_{(k+1)(w+1)}$ in $G^{k}_{w,p}$ contains $\binom{(k+1)(w+1)}{2}$ edges. Each of the $2^{m}$ copies of $G^{k}_{w,p_{w,k}}$ in $G^{k}_{w,p}$ contains $\binom{\lceil\delta_{k}(w)\rceil}{2} \leq (8k+16)^{2}(\frac{3}{2})^{2w}$ edges by \Cref{fact: deltak upper bound}. The contribution of the `downward' edges from each $K_{(k+1)(w+1)}$ (see \Cref{fig: example of Gkwp}) is the final term in:
    \begin{align*}
        |E(G^{k}_{w,p})|&\leq(2^{m}-1) 
        \binom{(k+1)(w+1)}{2}+2^{m}\cdot(8k+16)^{2}\Big(\frac{3}{2}\Big)^{2w}\\
        &+\sum_{i=1}^{m}2^{i}(k+1)(w+1)|V(G^{k}_{w,p-i})|\;.
    \end{align*}
    By \Cref{lem: general estimate for vertices in Gkwp} it follows that, for each $i\in\{1,\ldots,m\}$, $2^{i}|V(G^{k}_{w,p-i})|\leq5\cdot2^{p}$. Therefore, 
    \begin{align*}
        |E(G^{k}_{w,p})|<2^{m} \binom{(k+1)(w+1)}{2} +2^{m}\cdot(8k+16)^{2}\Big(\frac{3}{2}\Big)^{2w}+5\cdot 2^{p}m(k+1)(w+1)\;.
    \end{align*}
    Since $m\leq p$ and $2^{m}\leq2^{p}\frac{1}{2k+4}(\frac{2}{3})^{w}$ by \Cref{fact: exponential of pwk},
    \begin{align*}
        |E(G^{k}_{w,p})|\leq\frac{2^{p}}{2k+4}\Big(\frac{2}{3}\Big)^{w} \binom{(k+1)(w+1)}{2} + 2^{p}\frac{(8k+16)^{2}}{2k+4}\Big(\frac{3}{2}\Big)^{w}+5\cdot 2^{p}p(k+1)(w+1)\;.
    \end{align*}
    Since $w=\lceil\log_{3/2}p\rceil$, $5k2^{p}pw$ is the dominating function in the above inequality. The lower order terms are $\bigo(k2^{p}p)$. It follows that
    \begin{align*}
        |E(G^{k}_{w,p})|\leq5k2^{p}pw+\bigo(k2^{p}p)\leq5k2^{p}p\log_{3/2} p+\bigo(k2^{p}p)
    \end{align*}
    as claimed.
\end{proof}

For the remainder of this section, let $p:=\lceil\log_{2}(2n)\rceil$ and consider $n$ large relative to $k$.

\begin{theorem}
The graph $G^{k}_{w,p}$ satisfies 
     $|V(G^{k}_{w,p})|\leq20n$ and  $|E(G^{k}_{w,p})|\leq 20kn(\log_{2} n)(\log_{3/2}\log_{2} n)+\bigo(kn\log n)$. Moreover, $G^{k}_{w,p}$ contains every graph with $n$ vertices and treewidth at most $k$.
\end{theorem}
\begin{proof}
    By \Cref{lem: general estimate for vertices in Gkwp}, $|V(G^{k}_{w,p})|\leq5\cdot 2^{\lceil\log_{2}(2n)\rceil}\leq20n$. By \Cref{lem: estimating edges in Gkwp},
    \begin{align*}
        |E(G^{k}_{w,p})|&\leq5k2^{\lceil\log_{2}(2n)\rceil}\cdot\lceil\log_{2}(2n)\rceil\cdot\log_{3/2}\lceil\log_{2}(2n)\rceil+\bigo(k2^{\lceil\log_{2}(2n)\rceil}\cdot\lceil\log_{2}(2n)\rceil)\\
        &\leq20kn(\log_{2} n)(\log_{3/2}\log_{2} n)+\bigo(kn\log n)\;.
    \end{align*}
    Furthermore, $\gamma(w,p)\geq n$ by the same calculation in \Cref{thm: estimating vertices and edges in Gwp in terms of n}. Therefore, \Cref{thm: construction proof for Gkwp} implies that $G^{k}_{w,p}$ contains every graph with $n$ vertices and treewidth at most $k$.
\end{proof}

\begin{corollary}
    $s_k(n)\leq\bigo(kn(\log n)(\log\log n))$.
\end{corollary}

\renewcommand\bibsection{\subsection*{\refname}} % References heading

\bibliographystyle{unsrtnat}
\bibliography{references} % References file

\end{document}